\documentclass[12pt,reqno]{amsart}

\usepackage{amsthm}
\usepackage{amssymb}
\usepackage{latexsym}
\usepackage{multicol}
\usepackage{verbatim,enumerate}
\usepackage{accents}
\usepackage{wasysym }
\usepackage[usenames]{color}
\usepackage{hyperref}
\usepackage{amsmath, amscd}
\usepackage{soul}
\usepackage{mathrsfs}
\usepackage{tikz}
\usetikzlibrary{matrix,arrows,decorations.pathmorphing}
\usetikzlibrary{positioning}

\advance\textwidth by 1.3in \advance\oddsidemargin by -.6in \advance\evensidemargin by -.6in
\advance\textwidth by 1.2in \advance\oddsidemargin by -.6in \advance\evensidemargin by -.6in

\theoremstyle{plain}
\newtheorem{prop}{Proposition}
\newtheorem{thm}{Theorem}
\newtheorem{lem}{Lemma}
\newtheorem{cor}{Corollary}

\theoremstyle{definition}
\newtheorem{example}{Example}
\newtheorem{defn}{Definition}
\newtheorem{theom}{Theorem}

\theoremstyle{remark}
\newtheorem*{rem}{Remark}

\newcommand{\lie}[1]{\mathfrak{#1}}

\newcommand\bz{\mathbb Z}

\DeclareMathOperator{\htt}{ht}
\DeclareMathOperator{\mult}{mult}
\DeclareMathOperator{\supp}{supp}
\DeclareMathOperator{\type}{type}

\newenvironment{pf}{\proof}{\endproof}
\newcounter{cnt}
 \makeatletter
\def\mydggeometry{\makeatletter\dg@YGRID=1\dg@XGRID=20\unitlength=0.003pt\makeatother}
\makeatother \theoremstyle{remark}

\numberwithin{equation}{section}
\makeatletter
\def\section{\def\@secnumfont{\mdseries}\@startsection{section}{1}%
  \z@{.7\linespacing\@plus\linespacing}{.5\linespacing}%
  {\normalfont\scshape\centering}}
\def\subsection{\def\@secnumfont{\bfseries}\@startsection{subsection}{2}%
  {\parindent}{.5\linespacing\@plus.7\linespacing}{-.5em}%
  {\normalfont\bfseries}}
\makeatother

\begin{document}


\title{Chromatic symmetric function of graphs from Borcherds algebras}


\author{G. Arunkumar}
\thanks{The author is grateful to the Indian Institute of Science Education and Research, Mohali for the institute postdoctoral fellowship.}
\address{Indian Institute of Science Education and Research, Mohali, India}
\email{arun.maths123@gmail.com, gakumar@iisermohali.ac.in.}

\subjclass [2010]{05C15, 05C31, 05E05, 05E15, 17B01, 17B67}
\keywords{Borcherds algebras, Weyl denominator identity, Chromatic symmetric functions, $G$-symmetric functions. Free partially commutative Lie algebras}

\begin{abstract}

Let $\lie g$ be a Borcherds algebra with the associated graph $G$. 
We prove that the chromatic symmetric function of $G$ can be recovered from the Weyl denominator identity of $\lie g$ and this gives a Lie theoretic proof of Stanley's expression for chromatic symmetric function in terms of power sum symmetric functions. Also, this gives an expression for the chromatic symmetric function of $G$ in terms of root multiplicities of $\lie g$. We prove a modified Weyl denominator identity for Borcherds algebras which is an extension of the celebrated classical Weyl denominator identity and this plays an important role in the proof our results. The absolute value of the linear coefficient of the chromatic polynomial of $G$ is known as the chromatic discriminant of $G$. As an application of our main theorem, we prove that certain coefficients appearing in the above said expression of  chromatic symmetric function is equal to the chromatic discriminant of $G$.  Also, we find a connection between the Weyl denominators and the $G$-elementary symmetric functions. Using this connection, we give a Lie theoretic proof of non-negativity of coefficients of $G$-power sum symmetric functions. 
\end{abstract}
\maketitle
\section{Introduction}

Borcherds algebras were introduced by R. Borcherds in \cite{Bor88} which are a natural generalization of Kac-Moody Lie algebras. 
Our definition of Borcherds algebras follows the approach of \cite{Ej96}, in particular, the simple roots $\alpha_i$, $i\in I$ are linearly independent. We are interested in the denominator identity of Borcherds algebras. The denominator identity of Lie algebras is appearing in many places in combinatorics. In the celebrated paper of Macdonald \cite{mac72},  the denominator identities of various Lie algebras are shown to have many interesting connections with important identities in Combinatorics. For example, in the case $\lie {sl}_n(\mathbb C)$, denominator identity is the same as the Vondermonde's determinant identity. In the case of affine Lie algebras, these are the Macdonald identities. See also \cite{lepo78,kang96}.

 In this paper, we prove that the denominator identities are related to the chromatic symmetric functions of graphs. Let $G$ be a simple connected graph with vertex set $I$ and edge set $E$. In \cite{S95}, Stanley introduced the chromatic symmetric function $X_G$ of $G$ as a symmetric function generalization of the chromatic polynomial of $G$:
 Let $\mathcal{P}(S)$ be the power set of a set $S$. For a tuple of non-negative integers $\mathbf{k}=(k_i: i\in I)$, we define $\mathrm{supp}(\bold k)=\{i\in I: k_i\neq 0\}$.
The $\bold k$-multicoloring of the graph $G$ is a generalization of the well-known graph coloring which is defined as follows: Let $\mathbf{k}=(k_i: i\in I)$ a tuple of non--negative integers such that $|\mathrm{supp}(\bold k)|<\infty$. We call a map $\tau: I\rightarrow \mathcal{P}\big(\mathbb{N}\big)$ a proper vertex $\bold k$-multicoloring of $G$ if the following conditions are satisfied:
 \begin{enumerate}
 	\item[(i)]  For all $i\in I$ we have $|\tau(i)|=k_i$, 
 	\item[(ii)] For all $i,j\in I$ such that $(i,j)\in E(G)$ we have $\tau(i)\cap \tau(j)=\emptyset$.
 \end{enumerate}
 
 In other words, no two adjacent vertices share the same color. 
 Note that $k_i=1$ for $i\in I$  corresponds to the classical graph coloring of the graph $G$. The number of $\bold k$-multicoloring of $G$ using $q$ colors is counted by the number of $\bold k$-multicolorings $\tau$ such that $\tau(i) \subset \{1,2,\dots,q\} \,\forall\, i \in I$. It is well-known that this number is a polynomial in $q$, called the generalized chromatic polynomial $\pi_\mathbf{k}^G(q)$ \cite{akv}. When $\bold k: = \bold 1 = (1,1,\dots,1)$ $\pi_{\bold k}^G(q)$ represents the classical chromatic polynomial of $G$.
 
 \begin{rem}\label{connection}
 	There is a close relationship between the ordinary chromatic polynomials and the generalized chromatic polynomials. We have
 	\begin{equation}\label{1=k}
 	\pi_\bold k^G(q)=\frac{1}{\bold k!}\pi_{\bold 1}^{G(\bold k)}(q)
 	\end{equation}
 	where $\pi_{\bold 1}^{G(\bold k)}(q) \ \text{is the chromatic polynomial of the graph
 		$G(\bold k)$}$ and $\bold k!=\prod_{i\in I}k_i!$. The graph $G(\bold k)$ is the join of $G$ with respect to $\bold k$ which is constructed as follows:
 	For each $j\in \mathrm{supp}(\bold k)$, take a clique (complete graph) of size $k_j$ with vertex set $\{j^1,\dots, j^{k_j}\}$ and join all vertices of the $r$--th and $s$--th cliques if 
 	$(r,s)\in E(G).$ For more details about the multicoloring of a graph, we refer to  \cite{akv,HMK04} .
 \end{rem}
 Given this notion of vertex multicoloring we have the following definition of $\bold k$-chromatic symmetric functions.
 \begin{defn}\cite{S95,S98}
 	Let $X_1, X_2,\dots$ be a collection of commuting indeterminates. Let $\tau$ be a $\bold k$-multicoloring of $G$, we define $X(\tau) = \prod_{i\,\in\, I} \Big(\prod_{m \,\in\, \tau(i)} X_m\Big)$. The $\bold k$-chromatic symmetric function, denoted by $X^{\bold k}_{G}$, is defined as follows: 
 	
 	\begin{equation} \label{k-chsym}
 	X_{G}^{\bold k}= \sum_{\substack{\tau \\ \bold k-\text{proper} \\ \text{coloring} }} X(\tau) = \sum_{\substack{\tau \\ \bold k-\text{proper} \\ \text{coloring} }} \prod_{i\,\in\, I} \Big(\prod_{m \,\in\, \tau(i)} X_m\Big)
 	\end{equation}
 	
 \end{defn}
 \begin{rem}
 	We observe that $$X_G^{\bold k}(\underbrace{1,\dots,1}_{q\text{-times}},0,0,\dots) = \pi_{\bold k}^G(q).$$
 	
 	Also, when $\bold k = \bold 1 =(1,1,\dots,1)$ the $\bold k$-chromatic symmetric function of $G$ reduces to the chromatic symmetric function introduced in \cite{S95}. In this case, the $\bold k$-chromatic symmetric function $X_G^{\bold k}$ is simply denoted by $X_G$. 
 \end{rem}

 We explore the relation between chromatic symmetric functions and the root multiplicities of Borcherds algebras. This way the existing combinatorial results concerning denominator identity and the root multiplicities of Borcherds algebras might shed light on the problems in the theory of chromatic symmetric functions.

Let $A=(a_{ij})_{i,j\in I}$ be a Borcherds-Cartan matrix indexed by a finite or countably infinite set $I$ (Section \ref{bcm}). Let $\lie g$ be the Borcherds algebra associated with $A$ (Section \ref{ba}). In \cite[Section 2.1]{wm95}, the following notion of quasi-Dynkin diagram $G$ of the Borcherds algebra $\lie g$ is defined. The graph $G$ has vertex set $I$, with an edge between two vertices $i$ and $j$ if, and only if, $a_{ij}\neq 0$ for $i,j\in I$, $i\neq j$.
Note that $G$ is a simple (finite or infinite) graph which we call simply the $\mathrm{graph \ of} \ \lie g$. \textit{In the sequence, whenever we talk about a graph $G$ the ambient Borcherds algebra is understood}.
A finite subset $S\subseteq I$ is said to be connected if the corresponding subgraph generated by $S$ is connected. A subset $S$ of the vertex set $I$ is said to be stable or independent if there exists no edge between any two elements of $S$. A maximum independent set is an independent set of largest possible size for $G$. The size of the maximum independent set in $G$ is called the independence number of $G$ and denoted by $\alpha(G)$.

In \cite{akv}, the relation between the $\bold k$-chromatic polynomial of $G$ and the root multiplicities of $\lie g$ has been studied [c.f. Section \ref{k-poly}].  In particular, the following theorem from \cite[Theorem 1]{akv} expresses the chromatic polynomial $\pi_{\bold 1}^G(q)$ of $G$ in terms of certain root multiplicities of $\lie g$. 

 \begin{thm}\label{chmply}
	\textit{	Let $G$ be the graph of a Borcherds algebra $\lie g$. Then	\begin{equation}\label{vv}
		\pi_{\bold 1}^G(q) = (-1)^{\htt(\eta(\bold 1)) } \sum_{\bold J \in L_G} (-1)^{|\bold J|} (\mult (\bold J)) \,q^{|\pi|},
		\end{equation}
		where $L_G$ is the bond lattice of $G$ [c.f. Definition \ref{bond}] and $\htt (\bold k) := \sum_{i \in I}k_i$} if $\bold k = (k_i:i \in I)$.
\end{thm}

The main tool used in the proof of Theorem \ref{chmply} is the following Weyl denominator identity of $\lie g$ from \cite{Bor88} and \cite[Theorem 3.16]{Ej96}. 
\begin{eqnarray}
U:=  \sum_{w \in W } (-1)^{\ell(w)} \sum_{ \gamma \in
	\Omega} (-1)^{\htt(\gamma)} e^{w(\rho -\gamma)-\rho } & =& \prod_{\alpha \in \Delta_+} (1 - e^{-\alpha})^{\dim \mathfrak{g}_{\alpha}} 
 \end{eqnarray}  
where $\Omega$ is the set of all $\gamma\in Q_+$ such that $\gamma$ is a finite sum of mutually orthogonal distinct imaginary simple roots. Note that $0\in \Omega$ and $\alpha\in \Omega,$ if $\alpha$ is an imaginary simple root.

Since the chromatic symmetric functions are a natural generalization of chromatic polynomials it is natural to ask for a connection between chromatic symmetric functions and Borcherds algebras. In this paper, we are interested in exploring this connection. For this reason, we generalize the above given denominator identity as follows. For an indeterminate $X$, we have
 \begin{eqnarray}
 U(X):=  \sum_{w \in W } (-1)^{\ell(w)} \sum_{ \gamma \in
 	\Omega} (-1)^{\htt(\gamma)} X^{-\htt(w(\rho -\gamma)-\rho)} e^{w(\rho -\gamma)-\rho } & =& \prod_{\alpha \in \Delta_+} (1 - X^{-\htt(\alpha)}e^{-\alpha})^{\dim \mathfrak{g}_{\alpha}} 
 \label{eq:mdenom} \end{eqnarray}  
 The proof of this identity is given in Proposition \ref{md}. We call this identity,  the modified Weyl denominator identity and $U(X)$ the modified Weyl denominator. We show that the chromatic symmetric function can be recovered from the modified Weyl denominator. We will prove the following expression for the chromatic symmetric function in terms of root multiplicities of the Borcherds algebra $\lie g$ which is an extension of Theorem \ref{chmply} to the case of chromatic symmetric functions.

 \begin{thm}\label{mainthm}
 	\textit{Let $G$ be the graph of a Borcherds algebra $\lie g$. 	Then
 			\begin{equation}\label{chmplyeq}
 			X_G = (-1)^{\htt (\eta(\bold 1)) }\sum_{\bold J \in L_G} (-1)^{|\bold J|}(\mult (\bold J)) \,p_{\type (\bold J)},
 			\end{equation}
 			where $L_G$ is the bond lattice of $G$.}
  \end{thm}
 In Example \ref{G4}, we have explained how the above result can be used to distinguish the graphs of order $4$ by their chromatic symmetric functions. This example shows that calculating certain root multiplicities might sufficient to show the given graphs are distinguished by their chromatic symmetric functions. The proof of this theorem and many motivating examples for this proof are given in Section \ref{mains}. As a corollary, using \cite[Proposition 1.4]{VV15}, we get a Lie theoretic proof of the following theorem of Stanley \cite[Theorem 2.6]{S95}.
 
 \begin{thm}\label{stan}
 	$$X_G = \sum_{\bold J \in L_G} \mu(\hat{0},\bold J) \,p_{\type (\bold J)},$$
 	where $L_G$ is the bond lattice of $G$.
 \end{thm} 

  In \cite{akv}, Theorem \ref{chmply} is proved for the more general $\bold k$-chromatic polynomials of $G$.  
Using the modified Weyl denominator identity, we will prove the following expression for the $\bold k$-chromatic symmetric functions in terms of root multiplicities of the Borcherds algebra $\lie g$ which is an extension of \cite[Theorem 1]{akv} to the case of chromatic symmetric functions.

\begin{thm}\label{mainkthm}
	\textit{	Let $G$ be the graph of a Borcherds algebra $\lie g$. For a fixed tuple of non-negative integers $\mathbf{k}=(k_i: i\in I)$ such that $k_i\leq 1$ for $i\in I^{\text{re}}$ 
		and $1<|\mathrm{supp}(\bold k)|<\infty$ we have
		\begin{equation}
		X^{\bold k}_G =(-1)^{\htt(\eta(\bold k))}\sum_{\substack{\bold J \in L_G(\bold k) \\  \bar{\bold J} = \{J_1,\dots,J_k\}}} (-1)^{|\bar{\bold J}|}\Bigg(\prod_{J \in \bar{\bold J} } \binom{\mult (\beta(J))}{D(J,\bold J)}\Bigg) p_{\type (\bold J)},
		\end{equation}
		where $\bar{\bold J}$ is the underlying set of the multiset $\bold J$.}
\end{thm}
The proof of Theorem \ref{mainkthm} is also given in Section \ref{mains}.

 The bond lattice $L_G$ consists of partitions of the vertex set $I$ in which each part induces a connected subgraph of $G$. We have assumed that $G$ is connected and so the vertex set $I$ itself is an element of $L_G$ and has partition type $(n)$ if $G$ has $n$ vertices. The absolute value of the linear coefficient of the chromatic polynomial of $G$ is known as the \emph{chromatic discriminant} of $G$ \cite{MR1861053,a2}. Equation \eqref{vv} implies that the chromatic discriminant of $G$ is equal to the dimension of the root space $\lie g_{\beta}$, where $\beta$ is the sum of all simple roots. Theorem \ref{mainthm} has the following corollary which gives the following interesting connection between the chromatic discriminant and the chromatic symmetric function of $G$. 
\begin{prop}\label{discriminant} The coefficient of $p_{(n)}$ in the chromatic symmetric function $X_G$ is equal to the chromatic discriminant of $G$. i.e.,
	$$X_G[p_{(n)}] = \text{ Chromatic discriminant of $G$}$$
\end{prop}
  We note that chromatic discriminant plays an important role in Example \ref{G4}. In Section \ref{G-syms}, we explore the connection between the denominator identity and the $G$-symmetric functions. The $G$-symmetric functions are introduced by Stanley in \cite{S98}. 
  \begin{defn}\label{G-symd}
  	  The $G$-analogue of the $i$th elementary symmetry function is defined as follows. $$e^{G}_i =  \sum_{S}\left(\prod_{\alpha \in S}e^{-\alpha}\right),$$ where $S$ ranges over all $i$-element stable subsets of the vertex set $I$ of $G$. 	Given a partition $\lambda = \lambda_1 \ge \lambda_2 \ge \dots \ge \lambda_k$, $k \in \mathbb{N}$, we define $e_{\lambda}^G= \prod_{i=1}^k e_{\lambda_i}^G$.
  \end{defn}
  
 Given this definition, we have the following result \cite[Proposition 2.1]{S98}  of Stanley which explains a connection between the chromatic symmetric functions and the $G$-elementary symmetric functions.
\begin{prop}\label{T=X}
	\textit{Let $$T(x,v) = \sum_{\substack{\lambda \\ \text{stable}}} m_{\lambda}(x) e_{\lambda}^G(v).$$ Then $$T(x,v) [e^{-\eta(\bold k)}] = X^{\bold k}_G.$$}
\end{prop}
We prove that the function $T(x,v)$ can be recovered from the modified Weyl denominator. This gives a different proof of the statement that the chromatic symmetric function can be recovered from the modified Weyl denominator.  Let $X_1,X_2,\dots$ be a commuting family of indeterminates and consider the modified Weyl denominators $U(X_i)$ associated with these indeterminates.
We prove that the function $T(x,v)$ can be recovered from $\prod_{i=1}^{\infty} U(X_i)$ using Lie theoretic ideas
which gives an alternate proof to the above proposition. This is done in Section \ref{G-syms}.

As an application, we give a Lie theoretic proof of the following theorem which talks about the non-negativity of the coefficients in the $G$-power sum symmetric functions. The following theorem is proved in \cite[Theorem 2.3]{S98}. 

\begin{thm}\label{power}
	The $G$-power sum symmetric function $p_{\lambda}^G$ is a  polynomial with non-negative integral coefficients.
\end{thm}
{\em Acknowledgements. The author is grateful to Sankaran Viswanath, Tanusree Khandai and R. Venkatesh for many helpful discussions and constant support.}

\section{Borcherds algebras and the modified denominator identity}\label{section2}
In this section, we give some definitions and results from the theory of Borcherds algebras (also called generalized Kac--Moody algebras) and prove the modified denominator identity (Proposition \ref{md}). Further details about the Borcherds algebras can be seen in  \cite{Bor88,Ej96,K90} and the references therein.
\subsection{Borcherds-Cartan matrix}\label{bcm}
 A real matrix  
$A=(a_{ij})_{i,j\in I}$ indexed by a finite or countably infinite set, which we identify with $I=\{1, \dots, n\}$ or $\mathbb{Z}_+$ respectively, is said to be a \textit{Borcherds--Cartan\ matrix}
if the following conditions are satisfied for all $i,j\in I$:
\begin{enumerate}
\item $A$ is symmetrizable
\item $a_{ii}=2$ or $a_{ii}\leq 0$
\item $a_{ij}\leq 0$ if $i\neq j$ and $a_{ij}\in\mathbb{Z}$ if $a_{ii}=2$
\item $a_{ij}=0$ if and only if $a_{ji}=0$.
\end{enumerate} 
Recall that a matrix $A$ is called symmetrizable if there exists a diagonal matrix $D=\mathrm{diag}(\epsilon_i,i\in I)$ with positive entries such that $DA$ is symmetric. Set $I^\text{re}=\{i\in I: a_{ii}=2\}$ and $I^{\text{im}}=I\backslash I^\text{re}$.  For example, the generalized Cartan matrices ($a_{ii} = 2$ for all $i \in I$ and $|I|<\infty$) defined in \cite[Chapter 1]{kac} is a Borcherds-Cartan matrix. If $A$ is the adjacency matrix of a graph $G$, then $-A$ is a Borcherds-Cartan matrix. For graph-theoretic terminologies, we follow \cite{diestel}.
\subsection{Borcherds algebra}\label{ba}
The Borcherds algebra $\lie g=\mathfrak{g}(A)$ associated with a Borcherds--Cartan matrix $A$ is the Lie algebra generated by $e_i, f_i, h_i$, $i\in I$ with the following defining relations:
\begin{enumerate}
 \item[(R1)] $[h_i, h_j]=0$ for $i,j\in I$
 \item[(R2)] $[h_i, e_k]=a_{ik}e_i$,  $[h_i, f_k]=-a_{ik}f_i$ for $i,k\in I$
 \item[(R3)] $[e_i, f_j]=\delta_{ij}h_i$ for $i, j\in I$
 \item[(R4)] $(\text{ad }e_i)^{1-a_{ij}}e_j=0$, $(\text{ad }f_i)^{1-a_{ij}}f_j=0$ if $i\in I^\text{re}$ and $i\neq j$
 \item[(R5)] $[e_i, e_j]=0$ and $[f_i, f_j]=0$ if $i, j \in I^\text{im}$ and $a_{ij}=0$.
\end{enumerate}

If $I = I^{\text{re}}$ and $|I|<\infty$ then $\lie g(A)$ is a Kac-Moody Lie algebra. In this case, for $i\in I$, the subalgebra spanned by the elements $h_i,e_i,f_i$ is isomorphic to $\mathfrak{sl}_2$ (possibly after rescaling $e_i$ and $h_i$). In case of Borcherds algebras, if $a_{ii}=0$, the subalgebra spanned by the elements $h_i,e_i,f_i$ is a Heisenberg algebra. This is one of the main structural difference between these algebras. See \cite[Proposition 1.5]{Ej96} for more details.

\subsection{The graph of a Borcherds algebra}\label{graph}
For a given Borcherds algebra $\lie g(A)$ we associate a graph $G$ as follows: $G$ has vertex set $I$, with an edge between two vertices $i$ and $j$ iff 
$a_{ij}\neq 0$ for $i,j\in I$, $i\neq j$. The edge set of $G$ is denoted by $E(G)$ and $(i, j)$ denotes the edge between the nodes $i$ and $j$. 
Note that $G$ is a simple (finite or countably infinite) graph which we call the $\mathrm{graph \ of} \ \lie g$. Conversely, given a graph $G$ (finite/countably infinite) there exists a Borcherds algebra $\lie g$ whose associated graph is $G$. One way of doing this is by considering the negative of the adjacency matrix of $G$ which is a Borcherds-Cartan matrix. This way we can construct a Borcherds algebra from the graph $G$ which we will refer as the Borcherds algebra associated with the graph $G$.

Let $\bold k = (k_i: i \in I) \in \mathbb Z^{I}$ be such that $k_i$ is non-zero only for finitely many elements of $I$. Define $\supp \bold k = \{i \in I : k_i \ne 0\}$. In this paper, we will be always working with $\bold k$ satisfying this finite support condition. Even though the graph $G$ can be infinite, we will be working with the subgraph generated by $\supp \bold k$ in $G$ which is finite. 

\subsection{Root space and triangular decomposition}\label{Root space and traigular decompositions} 
Define a $\mathbb{Z}^{I}$--grading on $\lie g$ by giving $h_i$ degree $(0,0,\dots)$, $e_i$ degree $(0,\dots,0,1,0,\dots)$ and $f_i$ degree $(0,\dots,0,-1,0,\dots)$ where $\pm 1$ appears at the $i$--th position. For a sequence $(n_1,n_2,\dots)$, we denote by $\lie g(n_1,n_2,\dots)$ the corresponding graded piece; note that $\lie g(n_1,n_2,\dots)=0$ unless finitely many of the $n_i$ are non--zero. Let $\mathfrak{h}$ be the abelian subalgebra spanned by the $h_i$, $i\in I$ and let $\mathfrak{E}$ the space of commuting derivations of $\lie g$ spanned by the $D_i$, $i\in I$, where $D_i$ denotes the derivation that acts on $\lie g(n_1,n_2,\dots)$ as multiplication by the scalar $n_i$. Note that the abelian subalgebra $\mathfrak{E}\ltimes \mathfrak{h}$ of $\mathfrak{E}\ltimes \mathfrak{g}$ acts by scalars on $\lie g(n_1,n_2,\dots)$ and we have a root space decomposition:
\begin{equation}\label{rootdec}\mathfrak{g}=\bigoplus _{\alpha \in (\mathfrak{E}\ltimes \mathfrak{h})^*}
\mathfrak{g}_{\alpha }, \ \mathrm{where} \ \mathfrak{g}_{\alpha }
:=\{ x\in \mathfrak{g}\ |\ [h, x]=\alpha(h) x \ 
\mathrm{for\ all}\ h\in \mathfrak{E}\ltimes \mathfrak{h} \}.\end{equation}
Define $\Pi=\{\alpha_i\}_{i\in I}\subset (\mathfrak{E}\ltimes\lie h)^{*}$ by $\alpha_j((D_k,h_i))=\delta_{k,j}+a_{ij}$ and set
$$Q:=\bigoplus _{i\in I}\mathbb{Z}\alpha _i,\ \ Q_+ :=\sum _{i\in I}\mathbb{Z}_{+}\alpha _i.$$
Denote by $\Delta :=\{ \alpha \in (\mathfrak{E}\ltimes\lie h)^*\backslash \{0\} \mid \mathfrak{g}_{\alpha }\neq 0\}$ the set of roots, and by $\Delta_+$ the set of roots which are non--negative integral linear combinations of the $\alpha_i^{'}s$, called the positive roots. The elements in $\Pi$ are called the simple roots; we call $\Pi^\mathrm{re}:=\{\alpha_i: i\in I^\mathrm{re}\}$ the set of real simple roots and $\Pi^\mathrm{im}=\Pi\backslash \Pi^\mathrm{re}$ the set of imaginary simple roots. We have $\Delta =\Delta_+ \sqcup - \Delta_+$ and 
$$\mathfrak{g}_0=\mathfrak{h},\ \ \lie g_\alpha=\lie g(n_1,n_2,\dots),\ \text{ if }\ \alpha=\sum_{i\in I} n_i\alpha_i\in \Delta.$$

Moreover, we have a triangular decomposition
$$\lie g\cong \lie n^{-}\oplus \lie h \oplus \lie n^+,$$
where $\lie n^{+}$ (resp. $\lie n^{-}$) is the free Lie algebra generated by $e_i,\ i\in I$ (resp. $f_i,\ i\in I$) with defining relations 
$$(\text{ad }e_i)^{1-a_{ij}}e_j=0\ (\text{resp. } (\text{ad }f_i)^{1-a_{ij}}f_j=0) \text{ for } i\in I^\text{re}, \ j\in I \text{ and } i\neq j$$
and
$$[e_i, e_j]=0\ (\text{resp. } [f_i, f_j]=0) \text{ for } i, j \in I^\text{im} \text{ and } a_{ij}=0.$$

Given \eqref{rootdec} we have
$$\lie n^{\pm}=\bigoplus_{\alpha \in \pm\Delta_{+}}
\mathfrak{g}_{\alpha}.$$
Finally, given $\gamma=\sum_{i\in I}n_i\alpha_i\in Q_+$ (only finitely many $n_i$ are non--zero), we set $\text{ht}(\gamma)=\sum\limits_{i\in I}n_i.$

 

\subsection{The Weyl group and the Weyl vector}
We denote by $R=Q\otimes_{\bz} \mathbb{R}$ the real vector space spanned by $\Delta$. There exists a symmetric bilinear form on $R$ given by $(\alpha_i, \alpha_j)=\epsilon_i a_{ij}$ for $i, j\in I.$ For $i\in I^{\text{re}}$, define the linear isomorphism $\bold{s}_i$ of $R$ by 
$$\bold{s}_i(\lambda)=\lambda-\lambda(h_i)\alpha_i=\lambda-2\frac{(\lambda,\alpha_i)}{(\alpha_i,\alpha_i)}\alpha_i,\ \ \lambda\in R.$$ The Weyl group $W$ of $\mathfrak{g}$ is the subgroup of $\mathrm{GL}(R)$ generated by the simple reflections $\bold{s}_i$, $i\in I^\mathrm{re}$. Note that $W$ is a Coxeter group with canonical generators $\bold{s}_i, i\in I^\mathrm{re}$ and the above bilinear form is $W$--invariant. We denote by $\ell(w)=\mathrm{min}\{k\in \mathbb{N}: w=\bold{s}_{i_1}\cdots \bold{s}_{i_k}\}$ the length of $w\in W$ and $w=\bold{s}_{i_1}\cdots \bold{s}_{i_k}$ with $k=\ell(w)$ is called a reduced expression. We denote by $\Delta^\mathrm{re}=W(\Pi^\mathrm{re})$ the set of real roots and $\Delta^\mathrm{im}=\Delta\backslash \Delta^\mathrm{re}$ the set of imaginary roots. Equivalently, a root $\alpha$ is imaginary if and only if $(\alpha, \alpha)\le 0$ and else real. We can extend $(.,.)$ to a symmetric form on $(\mathfrak{E}\ltimes \lie h)^*$ satisfying $(\lambda,\alpha_i)=\lambda(\epsilon_ih_i)$ and also $\bold{s}_i$ to a linear isomorphism of $(\mathfrak{E}\ltimes \lie h)^*$ by 
$$\bold{s}_i(\lambda)=\lambda-\lambda(h_i)\alpha_i,\ \ \lambda\in (\mathfrak{E}\ltimes \lie h)^*.$$
Let $\rho$ be any element of $(\mathfrak{E}\ltimes \lie h)^*$ satisfying $2(\rho,\alpha_i)=(\alpha_i,\alpha_i)$ for all $i\in I$. Note that the Weyl vector $\rho$ of $\lie g$ need not be unique. The following lemma from \cite[Lemma 2.3]{akv} will be helpful on multiple occasions.
\begin{lem}\label{simplem}
	Let $i \in I^{im}$ and $\alpha \in \Delta_+ \backslash \{\alpha_i\}$ such that $\alpha(h_i)<0$. Then $\alpha + j \alpha_i \in \Delta_+$ for all $j \in \mathbb Z_+$.
\end{lem}

\subsection{The modified Weyl denominator identity} \label{mdeno}
The following Weyl denominator identity has been proved in \cite{Bor88}. See also \cite[Theorem 3.16]{Ej96}.
\begin{eqnarray}\label{denominator}
U:=  \sum_{w \in W } (-1)^{\ell(w)} \sum_{ \gamma \in
\Omega} (-1)^{\htt(\gamma)} e^{w(\rho -\gamma)-\rho } & =& \prod_{\alpha \in \Delta_+} (1 - e^{-\alpha})^{\dim \mathfrak{g}_{\alpha}} 
\label{eq:denom} \end{eqnarray}  
where $\Omega$ is the set of all $\gamma\in Q_+$ such that $\gamma$ is a finite sum of mutually orthogonal distinct imaginary simple roots. Note that $0\in \Omega$ and $\alpha\in \Omega,$ if $\alpha$ is an imaginary simple root. 

The following lemma from \cite[Lemma 3.6]{akv} will be helpful to do the calculations with the sum side of the denominator identity.
\begin{lem}\label{helplem}
	Let $w\in W$ and $\gamma\in \Omega$. We write $\rho-w(\rho)+w(\gamma)=\sum_{\alpha\in\Pi} b_{\alpha}(w,\gamma)\alpha$. Then we have
	\begin{enumerate}
		\item[(i)] $b_{\alpha}(w,\gamma)\in \mathbb{Z}_{+}$ for all $\alpha\in \Pi$ and $b_{\alpha}(w,\gamma)= 0$ if $\alpha\notin I(w)\cup I(\gamma)$,
		\item[(ii)] $I(w)=\{\alpha\in \Pi^{\mathrm re} : b_{\alpha}(w,\gamma)\geq 1\}$ and  $b_{\alpha}(w,\gamma)=1$ if $\alpha\in I(\gamma)$, and
		\item[(iii)] If $I(w) \cup I(\gamma)$ is independent if, and only if, $b_{\alpha}(w,\gamma)=1$ for all $\alpha\in I(w)\cup I(\gamma)$.
	
	\end{enumerate}
\end{lem}
	In the following proposition, we state and prove the modified denominator identity.
\begin{prop}\label{md}
Let $X$ be an indeterminate. Then we have
	\begin{eqnarray}\label{mdenom}
	U(X):=  \sum_{w \in W } (-1)^{\ell(w)} \sum_{ \gamma \in
		\Omega} (-1)^{\htt(\gamma)} X^{-\htt(w(\rho -\gamma)-\rho)} e^{w(\rho -\gamma)-\rho } & =& \prod_{\alpha \in \Delta_+} (1 - X^{-\htt(\alpha)}e^{-\alpha})^{\dim \mathfrak{g}_{\alpha}} 
	\label{eq:mdenom} \end{eqnarray}  
\end{prop}
\begin{pf}
Let $X_{\alpha} = e^{-\alpha}$. Then with the notations as in Lemma \ref{helplem}, we have $$e^{w(\rho-\gamma)-\rho} = \prod_{\alpha \in \pi}e^{b_{\alpha}(w,\gamma)\alpha} = \prod_{\alpha \in I(w) \cup I(\gamma)}X_{\alpha}^{b_{\alpha}(w,\gamma)} \in \mathbb C[X_\alpha : \alpha \in I].$$ This shows that $U=  \sum_{w \in W } (-1)^{\ell(w)} \sum_{ \gamma \in	\Omega} (-1)^{\htt(\gamma)} e^{w(\rho -\gamma)-\rho } \in \mathbb C[[X_\alpha : \alpha \in I]]$. Now, $$X^{-\htt(w(\rho -\gamma)-\rho)} e^{w(\rho -\gamma)-\rho } = \prod_{\alpha \in \pi}(XX_{\alpha})^{b_{\alpha}(w,\gamma)} \in \mathbb \mathbb C[X,X_\alpha : \alpha \in I].$$
This shows that the sum side of Equation \eqref{mdenom} can be obtained from the sum side of Equation \eqref{denominator} by a simple change of variable. Since in the product side of Equation \eqref{mdenom} also the same change of variable is applied, we get the required result from Equation \ref{denominator}.
\end{pf}

Next, we define the weighted bond lattice of the graph $G$.

\begin{defn}\label{bond}Let $L_{G}(\bold k)$ be the weighted bond lattice of $G$, which is the set of 
	$\mathbf{J}=\{J_1,\dots,J_k\}$  satisfying the following properties:
	\begin{enumerate}
		\item[(i)] $\bold J$ is a multiset, i.e. we allow $J_i=J_j$ for $i\neq j$
		
		\item[(ii)] each $J_i$ is a multiset and the subgraph spanned by the underlying set of $J_i$ is a connected subgraph of $G$ for each $1\le i\le k$ and
		
		\item[(iii)] the disjoint union $J_1\dot{\cup} \cdots \dot{\cup} J_k=\{\underbrace{\alpha_i,\dots, \alpha_i}_{k_i\, \mathrm{times}}: i\in I\}$.
		
	\end{enumerate}

\end{defn}
	When $\bold k = \bold 1 := (1,1,\dots)$ we denote $L_G(\bold k)$ simply by $L_G$. For $\mathbf{J}\in L_{G}(\bold k)$ we denote by $D(J_i,\mathbf{J})$ the number of times $J_i$ appear in $\mathbf{J}$.

\begin{example}
The Hasse diagram of the the bond lattice of weight $\bold k = (2,1,2)$ of the path graph on three vertices $\{0,1,2\}$ is given below [c.f. Example \ref{multi}].	
\begin{center}

	\tikzset{every picture/.style={line width=0.75pt}} 
	
	\begin{tikzpicture}[x=0.75pt,y=0.75pt,yscale=-1,xscale=1]
	
	\draw    (330.5,42) -- (240.34,80.22) ;
	\draw [shift={(238.5,81)}, rotate = 337.03] [color={rgb, 255:red, 0; green, 0; blue, 0 }  ][line width=0.75]    (10.93,-3.29) .. controls (6.95,-1.4) and (3.31,-0.3) .. (0,0) .. controls (3.31,0.3) and (6.95,1.4) .. (10.93,3.29)   ;
	\draw    (330.5,42) -- (308.48,81.26) ;
	\draw [shift={(307.5,83)}, rotate = 299.29] [color={rgb, 255:red, 0; green, 0; blue, 0 }  ][line width=0.75]    (10.93,-3.29) .. controls (6.95,-1.4) and (3.31,-0.3) .. (0,0) .. controls (3.31,0.3) and (6.95,1.4) .. (10.93,3.29)   ;
	\draw    (330.5,42) -- (358.33,80.38) ;
	\draw [shift={(359.5,82)}, rotate = 234.06] [color={rgb, 255:red, 0; green, 0; blue, 0 }  ][line width=0.75]    (10.93,-3.29) .. controls (6.95,-1.4) and (3.31,-0.3) .. (0,0) .. controls (3.31,0.3) and (6.95,1.4) .. (10.93,3.29)   ;
	\draw    (330.5,42) -- (428.64,80.27) ;
	\draw [shift={(430.5,81)}, rotate = 201.31] [color={rgb, 255:red, 0; green, 0; blue, 0 }  ][line width=0.75]    (10.93,-3.29) .. controls (6.95,-1.4) and (3.31,-0.3) .. (0,0) .. controls (3.31,0.3) and (6.95,1.4) .. (10.93,3.29)   ;
	\draw    (233.5,103) -- (233.5,157) ;
	\draw [shift={(233.5,159)}, rotate = 270] [color={rgb, 255:red, 0; green, 0; blue, 0 }  ][line width=0.75]    (10.93,-3.29) .. controls (6.95,-1.4) and (3.31,-0.3) .. (0,0) .. controls (3.31,0.3) and (6.95,1.4) .. (10.93,3.29)   ;
	\draw    (233.5,103) -- (164.08,156.78) ;
	\draw [shift={(162.5,158)}, rotate = 322.24] [color={rgb, 255:red, 0; green, 0; blue, 0 }  ][line width=0.75]    (10.93,-3.29) .. controls (6.95,-1.4) and (3.31,-0.3) .. (0,0) .. controls (3.31,0.3) and (6.95,1.4) .. (10.93,3.29)   ;
	\draw    (303.5,104) -- (304.46,154) ;
	\draw [shift={(304.5,156)}, rotate = 268.9] [color={rgb, 255:red, 0; green, 0; blue, 0 }  ][line width=0.75]    (10.93,-3.29) .. controls (6.95,-1.4) and (3.31,-0.3) .. (0,0) .. controls (3.31,0.3) and (6.95,1.4) .. (10.93,3.29)   ;
	\draw    (370.5,102) -- (370.5,155) ;
	\draw [shift={(370.5,157)}, rotate = 270] [color={rgb, 255:red, 0; green, 0; blue, 0 }  ][line width=0.75]    (10.93,-3.29) .. controls (6.95,-1.4) and (3.31,-0.3) .. (0,0) .. controls (3.31,0.3) and (6.95,1.4) .. (10.93,3.29)   ;
	\draw    (429,104) -- (429.48,158) ;
	\draw [shift={(429.5,160)}, rotate = 269.49] [color={rgb, 255:red, 0; green, 0; blue, 0 }  ][line width=0.75]    (10.93,-3.29) .. controls (6.95,-1.4) and (3.31,-0.3) .. (0,0) .. controls (3.31,0.3) and (6.95,1.4) .. (10.93,3.29)   ;
	\draw    (429,104) -- (494.94,156.75) ;
	\draw [shift={(496.5,158)}, rotate = 218.66] [color={rgb, 255:red, 0; green, 0; blue, 0 }  ][line width=0.75]    (10.93,-3.29) .. controls (6.95,-1.4) and (3.31,-0.3) .. (0,0) .. controls (3.31,0.3) and (6.95,1.4) .. (10.93,3.29)   ;
	\draw    (429,104) -- (235.43,158.46) ;
	\draw [shift={(233.5,159)}, rotate = 344.28999999999996] [color={rgb, 255:red, 0; green, 0; blue, 0 }  ][line width=0.75]    (10.93,-3.29) .. controls (6.95,-1.4) and (3.31,-0.3) .. (0,0) .. controls (3.31,0.3) and (6.95,1.4) .. (10.93,3.29)   ;
	\draw    (370.5,102) -- (306.05,154.73) ;
	\draw [shift={(304.5,156)}, rotate = 320.71000000000004] [color={rgb, 255:red, 0; green, 0; blue, 0 }  ][line width=0.75]    (10.93,-3.29) .. controls (6.95,-1.4) and (3.31,-0.3) .. (0,0) .. controls (3.31,0.3) and (6.95,1.4) .. (10.93,3.29)   ;
	\draw    (233.5,103) -- (368.64,156.27) ;
	\draw [shift={(370.5,157)}, rotate = 201.51] [color={rgb, 255:red, 0; green, 0; blue, 0 }  ][line width=0.75]    (10.93,-3.29) .. controls (6.95,-1.4) and (3.31,-0.3) .. (0,0) .. controls (3.31,0.3) and (6.95,1.4) .. (10.93,3.29)   ;
	\draw    (370.5,102) -- (428.07,158.6) ;
	\draw [shift={(429.5,160)}, rotate = 224.51] [color={rgb, 255:red, 0; green, 0; blue, 0 }  ][line width=0.75]    (10.93,-3.29) .. controls (6.95,-1.4) and (3.31,-0.3) .. (0,0) .. controls (3.31,0.3) and (6.95,1.4) .. (10.93,3.29)   ;
	\draw    (303.5,104) -- (494.57,157.46) ;
	\draw [shift={(496.5,158)}, rotate = 195.63] [color={rgb, 255:red, 0; green, 0; blue, 0 }  ][line width=0.75]    (10.93,-3.29) .. controls (6.95,-1.4) and (3.31,-0.3) .. (0,0) .. controls (3.31,0.3) and (6.95,1.4) .. (10.93,3.29)   ;
	\draw    (160.5,175) -- (303.63,228.3) ;
	\draw [shift={(305.5,229)}, rotate = 200.43] [color={rgb, 255:red, 0; green, 0; blue, 0 }  ][line width=0.75]    (10.93,-3.29) .. controls (6.95,-1.4) and (3.31,-0.3) .. (0,0) .. controls (3.31,0.3) and (6.95,1.4) .. (10.93,3.29)   ;
	\draw    (303.5,175) -- (304.46,228) ;
	\draw [shift={(304.5,230)}, rotate = 268.96] [color={rgb, 255:red, 0; green, 0; blue, 0 }  ][line width=0.75]    (10.93,-3.29) .. controls (6.95,-1.4) and (3.31,-0.3) .. (0,0) .. controls (3.31,0.3) and (6.95,1.4) .. (10.93,3.29)   ;
	\draw    (430.5,174) -- (370.46,230.63) ;
	\draw [shift={(369,232)}, rotate = 316.68] [color={rgb, 255:red, 0; green, 0; blue, 0 }  ][line width=0.75]    (10.93,-3.29) .. controls (6.95,-1.4) and (3.31,-0.3) .. (0,0) .. controls (3.31,0.3) and (6.95,1.4) .. (10.93,3.29)   ;
	\draw    (369.5,175) -- (429.98,226.7) ;
	\draw [shift={(431.5,228)}, rotate = 220.53] [color={rgb, 255:red, 0; green, 0; blue, 0 }  ][line width=0.75]    (10.93,-3.29) .. controls (6.95,-1.4) and (3.31,-0.3) .. (0,0) .. controls (3.31,0.3) and (6.95,1.4) .. (10.93,3.29)   ;
	\draw    (499.5,175) -- (431.11,225.81) ;
	\draw [shift={(429.5,227)}, rotate = 323.39] [color={rgb, 255:red, 0; green, 0; blue, 0 }  ][line width=0.75]    (10.93,-3.29) .. controls (6.95,-1.4) and (3.31,-0.3) .. (0,0) .. controls (3.31,0.3) and (6.95,1.4) .. (10.93,3.29)   ;
	\draw    (430.5,174) -- (429.54,226) ;
	\draw [shift={(429.5,228)}, rotate = 271.06] [color={rgb, 255:red, 0; green, 0; blue, 0 }  ][line width=0.75]    (10.93,-3.29) .. controls (6.95,-1.4) and (3.31,-0.3) .. (0,0) .. controls (3.31,0.3) and (6.95,1.4) .. (10.93,3.29)   ;
	\draw    (369.5,175) -- (307.03,227.71) ;
	\draw [shift={(305.5,229)}, rotate = 319.84000000000003] [color={rgb, 255:red, 0; green, 0; blue, 0 }  ][line width=0.75]    (10.93,-3.29) .. controls (6.95,-1.4) and (3.31,-0.3) .. (0,0) .. controls (3.31,0.3) and (6.95,1.4) .. (10.93,3.29)   ;
	\draw    (234.5,173) -- (367.17,231.2) ;
	\draw [shift={(369,232)}, rotate = 203.69] [color={rgb, 255:red, 0; green, 0; blue, 0 }  ][line width=0.75]    (10.93,-3.29) .. controls (6.95,-1.4) and (3.31,-0.3) .. (0,0) .. controls (3.31,0.3) and (6.95,1.4) .. (10.93,3.29)   ;
	\draw    (369.5,175) -- (369.02,230) ;
	\draw [shift={(369,232)}, rotate = 270.5] [color={rgb, 255:red, 0; green, 0; blue, 0 }  ][line width=0.75]    (10.93,-3.29) .. controls (6.95,-1.4) and (3.31,-0.3) .. (0,0) .. controls (3.31,0.3) and (6.95,1.4) .. (10.93,3.29)   ;
	\draw    (160.5,175) -- (232.92,230.78) ;
	\draw [shift={(234.5,232)}, rotate = 217.61] [color={rgb, 255:red, 0; green, 0; blue, 0 }  ][line width=0.75]    (10.93,-3.29) .. controls (6.95,-1.4) and (3.31,-0.3) .. (0,0) .. controls (3.31,0.3) and (6.95,1.4) .. (10.93,3.29)   ;
	\draw    (234.5,173) -- (234.5,230) ;
	\draw [shift={(234.5,232)}, rotate = 270] [color={rgb, 255:red, 0; green, 0; blue, 0 }  ][line width=0.75]    (10.93,-3.29) .. controls (6.95,-1.4) and (3.31,-0.3) .. (0,0) .. controls (3.31,0.3) and (6.95,1.4) .. (10.93,3.29)   ;
	\draw    (236.5,250) -- (325.72,296.08) ;
	\draw [shift={(327.5,297)}, rotate = 207.32] [color={rgb, 255:red, 0; green, 0; blue, 0 }  ][line width=0.75]    (10.93,-3.29) .. controls (6.95,-1.4) and (3.31,-0.3) .. (0,0) .. controls (3.31,0.3) and (6.95,1.4) .. (10.93,3.29)   ;
	\draw    (307.5,249) -- (326.73,295.15) ;
	\draw [shift={(327.5,297)}, rotate = 247.38] [color={rgb, 255:red, 0; green, 0; blue, 0 }  ][line width=0.75]    (10.93,-3.29) .. controls (6.95,-1.4) and (3.31,-0.3) .. (0,0) .. controls (3.31,0.3) and (6.95,1.4) .. (10.93,3.29)   ;
	\draw    (371.5,249) -- (328.85,295.53) ;
	\draw [shift={(327.5,297)}, rotate = 312.51] [color={rgb, 255:red, 0; green, 0; blue, 0 }  ][line width=0.75]    (10.93,-3.29) .. controls (6.95,-1.4) and (3.31,-0.3) .. (0,0) .. controls (3.31,0.3) and (6.95,1.4) .. (10.93,3.29)   ;
	\draw    (429.5,249) -- (329.31,296.15) ;
	\draw [shift={(327.5,297)}, rotate = 334.8] [color={rgb, 255:red, 0; green, 0; blue, 0 }  ][line width=0.75]    (10.93,-3.29) .. controls (6.95,-1.4) and (3.31,-0.3) .. (0,0) .. controls (3.31,0.3) and (6.95,1.4) .. (10.93,3.29)   ;
	\draw    (223.5,82) -- (223.5,100) ;
	\draw    (347.5,297) -- (347.5,315) ;
	\draw    (334.5,297) -- (334.5,315) ;
	\draw    (362.5,82) -- (362.5,100) ;
	\draw    (315.5,82) -- (315.5,100) ;
	\draw    (242.5,156) -- (242.5,174) ;
	\draw    (222.5,156) -- (222.5,174) ;
	\draw    (175.5,156) -- (175.5,174) ;
	\draw    (145.5,156) -- (145.5,174) ;
	\draw    (433.5,82) -- (433.5,100) ;
	\draw    (241.5,230) -- (241.5,248) ;
	\draw    (222.5,230) -- (222.5,248) ;
	\draw    (519.5,157) -- (519.5,175) ;
	\draw    (507.5,157) -- (507.5,175) ;
	\draw    (435.5,156) -- (435.5,174) ;
	\draw    (425.5,155) -- (425.5,173) ;
	\draw    (365.5,156) -- (365.5,174) ;
	\draw    (352.5,156) -- (352.5,174) ;
	\draw    (319.5,156) -- (319.5,174) ;
	\draw    (299.5,156) -- (299.5,174) ;
	\draw    (320.5,297) -- (320.5,315) ;
	\draw    (307.5,297) -- (307.5,315) ;
	\draw    (448.5,231) -- (448.5,249) ;
	\draw    (437.5,231) -- (437.5,249) ;
	\draw    (424.5,230) -- (424.5,248) ;
	\draw    (378.5,231) -- (378.5,249) ;
	\draw    (365.5,231) -- (365.5,249) ;
	\draw    (353.5,230) -- (353.5,248) ;
	\draw    (323.5,230) -- (323.5,248) ;
	\draw    (302.5,230) -- (302.5,248) ;
	\draw    (289.5,230) -- (289.5,248) ;
	\draw    (252.5,230) -- (252.5,248) ;
	\draw    (303.5,104) -- (164.37,157.28) ;
	\draw [shift={(162.5,158)}, rotate = 339.03999999999996] [color={rgb, 255:red, 0; green, 0; blue, 0 }  ][line width=0.75]    (10.93,-3.29) .. controls (6.95,-1.4) and (3.31,-0.3) .. (0,0) .. controls (3.31,0.3) and (6.95,1.4) .. (10.93,3.29)   ;
	
	\draw (340,160) node [anchor=north west][inner sep=0.75pt]   [align=left] {0 0 122};
	\draw (279,160) node [anchor=north west][inner sep=0.75pt]   [align=left] {00 12 2};
	\draw (479,160) node [anchor=north west][inner sep=0.75pt]   [align=left] {001 2 2};
	\draw (403,160) node [anchor=north west][inner sep=0.75pt]   [align=left] {00 1 22};
	\draw (307,22) node [anchor=north west][inner sep=0.75pt]   [align=left] {00122};
	\draw (211.5,85) node [anchor=north west][inner sep=0.75pt]   [align=left] {0 0122};
	\draw (280,85) node [anchor=north west][inner sep=0.75pt]   [align=left] {0012 2};
	\draw (342,85) node [anchor=north west][inner sep=0.75pt]   [align=left] {00 122};
	\draw (406,85) node [anchor=north west][inner sep=0.75pt]   [align=left] {001 22};
	\draw (210,160) node [anchor=north west][inner sep=0.75pt]   [align=left] {0 01 22};
	\draw (133,160) node [anchor=north west][inner sep=0.75pt]   [align=left] {0 012 2};
	\draw (340,234) node [anchor=north west][inner sep=0.75pt]   [align=left] {0 0 1 22};
	\draw (278,234) node [anchor=north west][inner sep=0.75pt]   [align=left] {0 0 12 2};
	\draw (209,234) node [anchor=north west][inner sep=0.75pt]   [align=left] {0 01 2 2};
	\draw (404,234) node [anchor=north west][inner sep=0.75pt]   [align=left] {00 1 2 2};
	\draw (296,301) node [anchor=north west][inner sep=0.75pt]   [align=left] {0 0 1 2 2};

	\end{tikzpicture}

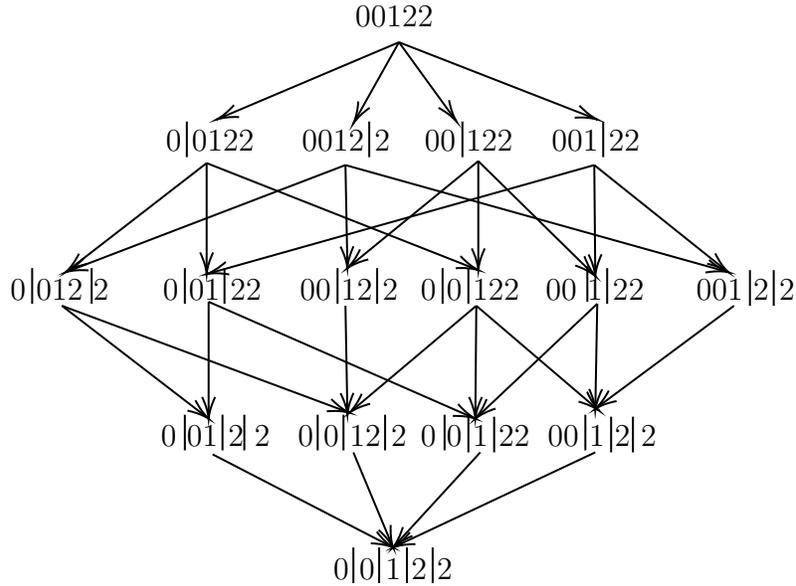
	Bond lattice of weight $\bold k = (2,1,2)$ of the path graph on $3$ vertices
\end{center}
\end{example}

We record the following lemma from \cite[Lemma 3.4]{akv} which will be needed later. 
\begin{lem}\label{bij}Assume that $k_i \le 1$ for $i \in I^{re}$ where $\bold k = (k_i: i \in I)$. Let $\mathcal{P}(\bold k)$ be the collection of sets $\gamma=\{\beta_1,\dots,\beta_r\}$ (we allow $\beta_i=\beta_j$ for $i\neq j$) such that each $\beta_i\in\Delta_+$ and $\beta_1+\dots+\beta_r=\eta(\bold k)$.
	Then the map $\Psi: L_{G}(\bold k)\rightarrow \mathcal{P}(\bold k)$ defined by $\{J_1,\dots,J_k\}\mapsto \{\beta(J_1),\dots,\beta(J_k)\}$ is a bijection.
\end{lem}
	\begin{rem}
Under the assumption of the above Lemma, by Lemma \ref{simplem}, each part $J_i$ of $\bold J \in L_G(\bold k)$ represents a positive root $\beta(J_i):=\sum_{\alpha\in J_i} \alpha \in \Delta_+$ .  
\end{rem}

\section{multicoloring and $\bold k$-chromatic symmetric  functions}\label{section3}
 Let $G$ be the graph of a Borcherds algebra $\lie g$. In this section, we give the examples of $\bold k$-multicoloring and prove an expression for the $\bold k$-chromatic symmetric functions.

\subsection{An expression for $\bold k$-chromatic polynomials}\label{k-poly}

\begin{example}\label{multi}
Let $G$ be the path graph on three vertices and let $\bold k=(2,1,2)$. We observe that $G$ cannot be $\bold k=(2,1,2)$-colored using two or fewer colors. We list the $\bold k=(2,1,2)$-colorings of $G$ using four colors, say blue, red, green and yellow, in which the first vertex receives the colors blue and red:

\noindent
\textbf{With three colors:}
\begin{center}
\hspace{-3cm}
\begin{tikzpicture}[scale=.6]
\draw (-1,0) node[anchor=east]{};
\draw[top color=blue,bottom color=red, xshift=3 cm,thick] (3 cm,0) circle(.3cm);
\draw[fill=green,xshift=4 cm,thick] (4 cm,0) circle(.3cm);
\draw[top color=blue,bottom color=red, xshift=7 cm,thick] (3 cm,0) circle (.3cm);
\foreach \y in {3.15,...,3.15}
\draw[xshift=\y cm,thick] (\y cm,0) -- +(1.4 cm,0);
\draw[xshift=8 cm,thick] (0: 3 mm) -- (0: 17 mm);
\end{tikzpicture}
\hspace{-1cm}
\begin{tikzpicture}[scale=.6]
\draw (-1,0) node[anchor=east]{};
\draw[top color=blue,bottom color=red, xshift=3 cm,thick] (3 cm,0) circle(.3cm);
\draw[fill=yellow,xshift=4 cm,thick] (4 cm,0) circle(.3cm);
\draw[top color=blue,bottom color=red, xshift=7 cm,thick] (3 cm,0) circle (.3cm);
\foreach \y in {3.15,...,3.15}
\draw[xshift=\y cm,thick] (\y cm,0) -- +(1.4 cm,0);
\draw[xshift=8 cm,thick] (0: 3 mm) -- (0: 17 mm);
\end{tikzpicture} \\
\end{center}
\textbf{With four colors:}
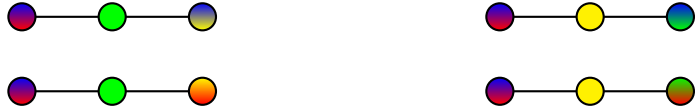
\begin{center}
\vspace{5mm}
\hspace{-3cm}
\begin{tikzpicture}[scale=.6]
\draw (-1,0) node[anchor=east]{};
\draw[top color=blue,bottom color=red, xshift=3 cm,thick] (3 cm,0) circle(.3cm);
\draw[fill=green,xshift=4 cm,thick] (4 cm,0) circle(.3cm);
\draw[top color=blue,bottom color=yellow, xshift=7 cm,thick] (3 cm,0) circle (.3cm);
\foreach \y in {3.15,...,3.15}
\draw[xshift=\y cm,thick] (\y cm,0) -- +(1.4 cm,0);
\draw[xshift=8 cm,thick] (0: 3 mm) -- (0: 17 mm);
\end{tikzpicture}
\hspace{-1cm}
\begin{tikzpicture}[scale=.6]
\draw (-1,0) node[anchor=east]{};
\draw[top color=blue,bottom color=red, xshift=3 cm,thick] (3 cm,0) circle(.3cm);
\draw[fill=yellow,xshift=4 cm,thick] (4 cm,0) circle(.3cm);
\draw[top color=blue,bottom color=green, xshift=7 cm,thick] (3 cm,0) circle (.3cm);
\foreach \y in {3.15,...,3.15}
\draw[xshift=\y cm,thick] (\y cm,0) -- +(1.4 cm,0);
\draw[xshift=8 cm,thick] (0: 3 mm) -- (0: 17 mm);
\end{tikzpicture} \\
\vspace{5mm}
\hspace{-3cm}
\begin{tikzpicture}[scale=.6]
\draw (-1,0) node[anchor=east]{};
\draw[top color=blue,bottom color=red, xshift=3 cm,thick] (3 cm,0) circle(.3cm);
\draw[fill=green,xshift=4 cm,thick] (4 cm,0) circle(.3cm);
\draw[top color=yellow,bottom color=red, xshift=7 cm,thick] (3 cm,0) circle (.3cm);
\foreach \y in {3.15,...,3.15}
\draw[xshift=\y cm,thick] (\y cm,0) -- +(1.4 cm,0);
\draw[xshift=8 cm,thick] (0: 3 mm) -- (0: 17 mm);
\end{tikzpicture}
\hspace{-1cm}
\begin{tikzpicture}[scale=.6]
\draw (-1,0) node[anchor=east]{};
\draw[top color=blue,bottom color=red, xshift=3 cm,thick] (3 cm,0) circle(.3cm);
\draw[fill=yellow,xshift=4 cm,thick] (4 cm,0) circle(.3cm);
\draw[top color=green,bottom color=red, xshift=7 cm,thick] (3 cm,0) circle (.3cm);
\foreach \y in {3.15,...,3.15}
\draw[xshift=\y cm,thick] (\y cm,0) -- +(1.4 cm,0);
\draw[xshift=8 cm,thick] (0: 3 mm) -- (0: 17 mm);
\end{tikzpicture}
\end{center}
\end{example}
We note that all possible $\bold k=(2,1,2)$ colorings of $G$ using the colors blue, red, green and yellow can be obtained from the above-given configurations. This is done by changing the choice of the colors in the first vertex and the colors in the remaining vertices according to the above configurations. Therefore the number of $\bold k=(2,1,2)$-colorings of the graph $G$ is equal to $\binom{4}{2} \cdot 6 = 36$. In Example \ref{mainexample}, we have calculated $\pi^G_{\bold k}(q)$ explicitly from which we see that $\pi^G_{\bold k}(4)$ is indeed equal to $36$. 


The generalized chromatic polynomial has the following well--known description. 
We denote by $P_k(\bold k,G)$ the set of all ordered $k$--tuples $(P_1,\dots,P_k)$ such that:
\begin{enumerate}
 \item[(i)] each $P_i$ is a non--empty independent subset of $I$, i.e. no two vertices have an edge between them; and

 \item[(ii)] the disjoint union of  $P_1,\dots, P_k$ is equal to the multiset $\{\underbrace{\alpha_i,\dots, \alpha_i}_{k_i\, \mathrm{times}}: i\in I \}$.
\end{enumerate} 
Then we have
\begin{equation}\label{defgenchr}\pi^G_\mathbf{k}(q)= \sum\limits_{k\ge0}|P_k(\bold k, G)| \, {q \choose k}.\end{equation}

\begin{rem}
	For a partition $P = (P_1,\dots,P_k) \in P_k(\bold k,G)$ its type is defined to be the composition $$(D(P_1, P),\dots,D(P_m, P))$$ where $D(P_i, P)$ is equal to the number of times $P_i$ appears in $P$ and $P_1,\dots,P_m$ are the distinct elements of $P$.
	
	For a partition (or a composition) $\lambda = (1^{a_1},2^{a_2},\dots)$ of $n$ we define $\lambda !$ to be the multinomial coefficient $\frac{n !}{ \prod\limits_{i \ge 1} a_i!}$.
	If we consider the elements of $P_k(\bold k,G)$ as multi-sets (instead of ordered tuples) then we have
	\begin{equation}\label{modifiedfefgenchr}
	\pi^G_\mathbf{k}(q)= \sum\limits_{k\ge0}\sum_{P \in P_k(\bold k, G)} \type(P)! {q \choose k}.
	\end{equation}
\end{rem}

\begin{example}\label{mainexample}
	Consider the path graph $G$ on three vertices. 
	\begin{center}
		\tikzset{every picture/.style={line width=0.75pt}} 
		
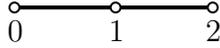
		\begin{tikzpicture}[x=0.75pt,y=0.75pt,yscale=-1,xscale=1]
		
		\draw [line width=1.5]    (85.5,118) -- (131.5,118) ;
		\draw   (80.5,118) .. controls (80.5,116.62) and (81.62,115.5) .. (83,115.5) .. controls (84.38,115.5) and (85.5,116.62) .. (85.5,118) .. controls (85.5,119.38) and (84.38,120.5) .. (83,120.5) .. controls (81.62,120.5) and (80.5,119.38) .. (80.5,118) -- cycle ;
		\draw   (131.5,118) .. controls (131.5,116.62) and (132.62,115.5) .. (134,115.5) .. controls (135.38,115.5) and (136.5,116.62) .. (136.5,118) .. controls (136.5,119.38) and (135.38,120.5) .. (134,120.5) .. controls (132.62,120.5) and (131.5,119.38) .. (131.5,118) -- cycle ;
		\draw [line width=1.5]    (136.5,118) -- (180.5,118) ;
		\draw   (180.5,118) .. controls (180.5,116.62) and (181.62,115.5) .. (183,115.5) .. controls (184.38,115.5) and (185.5,116.62) .. (185.5,118) .. controls (185.5,119.38) and (184.38,120.5) .. (183,120.5) .. controls (181.62,120.5) and (180.5,119.38) .. (180.5,118) -- cycle ;
		
		\draw (78,123) node [anchor=north west][inner sep=0.75pt]   [align=left] {0};
		\draw (129,123) node [anchor=north west][inner sep=0.75pt]   [align=left] {1};
		\draw (178,123) node [anchor=north west][inner sep=0.75pt]   [align=left] {2};
\end{tikzpicture}

	\end{center}

Then the graph $G(\bold k)$, the $\bold k=(2,1,2)$-join  of $G$ is given by

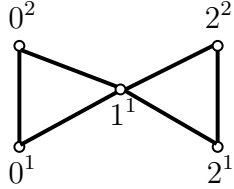
\begin{center}

	\tikzset{every picture/.style={line width=0.75pt}} 
	
	\begin{tikzpicture}[x=0.75pt,y=0.75pt,yscale=-1,xscale=1]
	
	\draw [line width=1.5]    (279.5,106) -- (328,79.5) ;
	\draw   (276.5,107) .. controls (276.5,105.62) and (277.62,104.5) .. (279,104.5) .. controls (280.38,104.5) and (281.5,105.62) .. (281.5,107) .. controls (281.5,108.38) and (280.38,109.5) .. (279,109.5) .. controls (277.62,109.5) and (276.5,108.38) .. (276.5,107) -- cycle ;
	\draw   (327.5,78) .. controls (327.5,76.62) and (328.62,75.5) .. (330,75.5) .. controls (331.38,75.5) and (332.5,76.62) .. (332.5,78) .. controls (332.5,79.38) and (331.38,80.5) .. (330,80.5) .. controls (328.62,80.5) and (327.5,79.38) .. (327.5,78) -- cycle ;
	\draw [line width=1.5]    (332,78.5) -- (378.5,105.5) ;
	\draw   (376.5,107) .. controls (376.5,105.62) and (377.62,104.5) .. (379,104.5) .. controls (380.38,104.5) and (381.5,105.62) .. (381.5,107) .. controls (381.5,108.38) and (380.38,109.5) .. (379,109.5) .. controls (377.62,109.5) and (376.5,108.38) .. (376.5,107) -- cycle ;
	\draw [line width=1.5]    (279,58.5) -- (279,104.5) ;
	\draw [line width=1.5]    (379,58.5) -- (379,104.5) ;
	\draw [line width=1.5]    (279.5,58) -- (328,76.5) ;
	\draw   (276.5,56) .. controls (276.5,54.62) and (277.62,53.5) .. (279,53.5) .. controls (280.38,53.5) and (281.5,54.62) .. (281.5,56) .. controls (281.5,57.38) and (280.38,58.5) .. (279,58.5) .. controls (277.62,58.5) and (276.5,57.38) .. (276.5,56) -- cycle ;
	\draw   (376.5,56) .. controls (376.5,54.62) and (377.62,53.5) .. (379,53.5) .. controls (380.38,53.5) and (381.5,54.62) .. (381.5,56) .. controls (381.5,57.38) and (380.38,58.5) .. (379,58.5) .. controls (377.62,58.5) and (376.5,57.38) .. (376.5,56) -- cycle ;
	\draw [line width=1.5]    (332.5,78) -- (376.5,56) ;
	
	\draw (272,110) node [anchor=north west][inner sep=0.75pt]   [align=left] {$\displaystyle 0^{1}$};
	\draw (323,81) node [anchor=north west][inner sep=0.75pt]   [align=left] {$\displaystyle 1^{1}$};
	\draw (372,110) node [anchor=north west][inner sep=0.75pt]   [align=left] {$\displaystyle 2^{1}$};
	\draw (272,32) node [anchor=north west][inner sep=0.75pt]   [align=left] {$\displaystyle  0^{2}$};
	\draw (371,32) node [anchor=north west][inner sep=0.75pt]   [align=left] {$\displaystyle 2^{2}$};

	\end{tikzpicture}

\end{center}
Now, upto permutation of the tuples, $$P_1(\bold k,G) = P_2(\bold k,G) = \phi,$$ $$P_3(\bold k,G) = \{(\{1\},\{0,2\},\{0,2\})\},$$
$$P_4(\bold k,G) = \{(\{1\},\{0\},\{2\},\{0,2\})\},$$
$$P_5(\bold k,G) = \{(\{1\},\{0\},\{0\},\{2\},\{2\})\}.$$
Therefore $|P_1(\bold k,G)| = |P_2(\bold k,G)| = 0$, $|P_3(\bold k,G)| = \frac{3!}{2}$, $|P_4(\bold k,G)| = 4!$, $|P_5(\bold k,G)| = \frac{5!}{2! 2!}$.
Therefore, by Equation \eqref{defgenchr}, we have 

$\pi_{\bold k}^G(q) = (\frac{1}{2})q(q-1)(q-2) + q(q-1)(q-2)(q-3) + (\frac{1}{4})q(q-1)(q-2)(q-3)(q-4) = (\frac{1}{4}) q(q-1)^2(q-2)^2.$

A direct computation shows that $$\pi_{\bold 1}^{G(\bold k)}(q) = q(q-1)^2(q-2)^2.$$
Also note that, $$\pi_{\bold k}^G(4) = 36 \text{	[c.f. Example \ref{multi}]}.$$

The above equalities explain Equation \eqref{defgenchr} and Equation \eqref{1=k}.

\end{example}



\subsection{An expression for $\bold k$-chromatic symmetric functions}\label{k-sym} 

We have the following expression for the $\bold k$-chromatic symmetric functions whose proof is immediate from the definition. We note that this is the extension of \eqref{defgenchr} to the case of chromatic symmetric functions.

\begin{prop}\label{eqdef}
  \begin{equation}\label{equiv}
X_{G}^{\bold k}= \sum_{k \ge 1} \sum_{\substack{P \in P_{k}(\bold k,G)\\ P = (P_1,P_2, \dots, P_k) }} \sum_{\substack{J \,\subseteq \mathbb{N}\\J = \{i_1,i_2,\dots,i_k\}}} x_{i_1}^{|P_1|} x_{i_2}^{|P_2|} \cdots x_{i_k}^{|P_k|}
\end{equation}
\end{prop}
Note that, substituting $x_i = \begin{cases}
 1 \text{ if } 1 \le i \le q\\ 0 \text{ otherwise}
\end{cases}$  in Equation \eqref{equiv} yields Equation \eqref{defgenchr}.

If we consider the elements of $P_k(\bold k,G)$ as multi-sets then we have
\begin{equation}\label{equivv}
X_G^{\bold k}= \sum\limits_{k\ge0}\sum_{\substack{P \in P_{k}(\bold k,G)\\ P = \{P_1,P_2, \dots, P_k\} }} \type(P)! \sum_{\substack{J \,\subseteq \mathbb{N}\\J = \{i_1,i_2,\dots,i_k\}}} x_{i_1}^{|P_1|} x_{i_2}^{|P_2|} \cdots x_{i_k}^{|P_k|}.
\end{equation} 

\begin{example}
	Let $G$ be the path graph on three vertices and let $\bold k = (2,1,2)$. From Example \ref{mainexample},  upto permutation of the tuples, $$P_1(\bold k,G) = P_2(\bold k,G) = \phi,$$ $$P_3(\bold k,G) = \{(\{1\},\{0,2\},\{0,2\})\},$$
	$$P_4(\bold k,G) = \{(\{1\},\{0\},\{2\},\{0,2\})\},$$
	$$P_5(\bold k,G) = \{(\{1\},\{0\},\{0\},\{2\},\{2\})\}.$$
	
	Now, from Equation \eqref{equivv}, we get 
	\begin{equation}
		X_G^{\bold k} = 3 \bold m_{221} +  24 \bold m_{2111} + 30 \bold m_{11111}.
	\end{equation}
\end{example}

\section{Denominator identity and the Chromatic symmetric functions}\label{mains}
In this section, we prove a connection between the root multiplicities of $\lie g$ and the $\bold k$-chromatic symmetric function of $G$. The main results of this section are Theorems \ref{mainthm} and \ref{mainkthm}. 
\subsection{The sum side of the denominator identity} In this subsection, we will prove that the chromatic symmetric function can be obtained from the sum side of modified Weyl denominator. 
\begin{example}\label{ssl3}
	As a motivating example, assume $\lie g = \lie{sl}_3(\mathbb C)$ the smallest non-trivial finite-dimensional simple Lie algebra for our purpose (The graph of $\lie{sl}_2(\mathbb C)$ is a single point and has no edges. So we are not considering this algebra).  This Lie algebra $\lie g$ consists of $3 \times 3$ traceless complex matrices. The Dynkin diagram $G$ of the algebra $\lie g$ is given below.

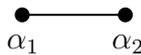
\begin{center}
		\tikzset{every picture/.style={line width=0.75pt}} 
	
	\begin{tikzpicture}[x=0.75pt,y=0.75pt,yscale=-1,xscale=1]
	
	\draw  [fill={rgb, 255:red, 0; green, 0; blue, 0 }  ,fill opacity=1 ] (367,201.25) .. controls (367,199.46) and (368.46,198) .. (370.25,198) .. controls (372.04,198) and (373.5,199.46) .. (373.5,201.25) .. controls (373.5,203.04) and (372.04,204.5) .. (370.25,204.5) .. controls (368.46,204.5) and (367,203.04) .. (367,201.25) -- cycle ;
	\draw  [fill={rgb, 255:red, 0; green, 0; blue, 0 }  ,fill opacity=1 ] (315,201.25) .. controls (315,199.46) and (316.46,198) .. (318.25,198) .. controls (320.04,198) and (321.5,199.46) .. (321.5,201.25) .. controls (321.5,203.04) and (320.04,204.5) .. (318.25,204.5) .. controls (316.46,204.5) and (315,203.04) .. (315,201.25) -- cycle ;
	\draw    (370.25,201.25) -- (318.25,201.25) ;
	
	\draw (362,210) node [anchor=north west][inner sep=0.75pt]   [align=left] {$\displaystyle \alpha _{2}$};
	\draw (309,210) node [anchor=north west][inner sep=0.75pt]   [align=left] {$\displaystyle \alpha _{1}$};

	\end{tikzpicture}
\end{center}

	The graph $G$ is the path graph on two vertices. In general, the Dynkin diagram of $\lie{sl}_n(\mathbb C)$ is the path graph on $n-1$ vertices. The set of simple roots of $\lie g$ is equal to $\Pi = \{\alpha_1=\epsilon_1-\epsilon_2,\alpha_2=\epsilon_2-\epsilon_3\}$ and the set of positive roots of $\lie g$ is $\Delta_+ = \{\alpha_1, \alpha_2, \alpha_1+\alpha_2\}$. The Weyl group $W$ is the symmetric group on three symbols $\{\epsilon_1,\epsilon_2,\epsilon_3\}$ permuting the indices. $\rho = \frac{1}{2}\sum_{ \alpha \in \Delta_+}\alpha = \epsilon_1-\epsilon_3$. This implies that $$U(X_i) = 1 - X_i e^{-\alpha_1} - X_i e^{-\alpha_2} + X_i^3 e^{-2\alpha_2-\alpha_1} + X_i^3 e^{-2\alpha_1 -\alpha_2} - X_i^4 e^{-2\alpha_1 -2\alpha_2} \text{	[c.f. Equation \eqref{mdenom}]}.$$
	We claim that, the coefficient of $e^{-\alpha_1-\alpha_2}$ in the product $\prod_{i = 1}^{\infty}U(X_i)$ is $X_G$. For this reason, we can assume $$U(X_i) = 1 - X_i e^{-\alpha_1} - X_i e^{-\alpha_2}.$$
	Consider the product:
	
	$$\prod_{i = 1}^{\infty}U(X_i) = (1 - X_1 e^{-\alpha_1} - X_1 e^{-\alpha_2})(1 - X_2 e^{-\alpha_1} - X_2 e^{-\alpha_2})\cdots(1 - X_k e^{-\alpha_1} - X_k e^{-\alpha_2})\cdots$$
	Then the coefficient of $e^{-\alpha_1-\alpha_2}$ in the above product is equal to $$\sum\limits_{\substack{(i_1,i_2) \in \mathbb N^2 \\ i_1 \ne i_2}}X_{i_1}X_{i_2} = 2! e_2 = X_G$$
	where $e_2$ is the second elementary symmetric function.
\end{example}
\begin{rem}
	If $G$ is the complete graph on $n$ vertices, then $X_G = n! e_n$ where $e_n$ is the $n$th elementary symmetric function.
\end{rem}
\begin{example}\label{ssl3h}
	In this example, we consider the untwisted affine Lie algebra $\lie g = \widehat{\lie{sl}_3}$. This is an infinite-dimensional Lie algebra.  Let $G$ be the Dynkin diagram of $\lie g$. Then $G$ is the cycle graph (complete graph) on three vertices:
	
	\begin{center}

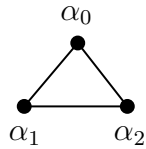
		\tikzset{every picture/.style={line width=0.75pt}} 
		
		\begin{tikzpicture}[x=0.75pt,y=0.75pt,yscale=-1,xscale=1]
		
		\draw  [fill={rgb, 255:red, 0; green, 0; blue, 0 }  ,fill opacity=1 ] (367,201.25) .. controls (367,199.46) and (368.46,198) .. (370.25,198) .. controls (372.04,198) and (373.5,199.46) .. (373.5,201.25) .. controls (373.5,203.04) and (372.04,204.5) .. (370.25,204.5) .. controls (368.46,204.5) and (367,203.04) .. (367,201.25) -- cycle ;
		\draw  [fill={rgb, 255:red, 0; green, 0; blue, 0 }  ,fill opacity=1 ] (315,201.25) .. controls (315,199.46) and (316.46,198) .. (318.25,198) .. controls (320.04,198) and (321.5,199.46) .. (321.5,201.25) .. controls (321.5,203.04) and (320.04,204.5) .. (318.25,204.5) .. controls (316.46,204.5) and (315,203.04) .. (315,201.25) -- cycle ;
		\draw    (370.25,201.25) -- (318.25,201.25) ;
		\draw  [fill={rgb, 255:red, 0; green, 0; blue, 0 }  ,fill opacity=1 ] (342,169.3) .. controls (341.97,167.51) and (343.4,166.03) .. (345.2,166) .. controls (346.99,165.97) and (348.47,167.4) .. (348.5,169.2) .. controls (348.53,170.99) and (347.1,172.47) .. (345.3,172.5) .. controls (343.51,172.53) and (342.03,171.1) .. (342,169.3) -- cycle ;
		\draw    (318.25,201.25) -- (345.25,169.25) ;
		\draw    (345.25,169.25) -- (370.25,201.25) ;
		
		\draw (362,210) node [anchor=north west][inner sep=0.75pt]   [align=left] {$\displaystyle \alpha _{2}$};
		\draw (309,210) node [anchor=north west][inner sep=0.75pt]   [align=left] {$\displaystyle \alpha _{1}$};
		\draw (335,149) node [anchor=north west][inner sep=0.75pt]   [align=left] {$\displaystyle \alpha _{0}$};

		\end{tikzpicture}
		
	\end{center} In general, the Dynkin diagram of $\widehat{\lie{sl}_n}$ is the cycle graph on $n$ vertices. Again, we are interested in the coefficient of $e^{-\alpha_0-\alpha_1-\alpha_2}$. For this reason, like in the case of $\lie{sl}_3$, we can assume that $$U(X_i) = 1 - X_i e^{-\alpha_0} - X_i e^{-\alpha_1}- X_i e^{-\alpha_2}.$$
	Then, the coefficient of $e^{-\alpha_0-\alpha_1-\alpha_2}$ in $\prod_{i = 1}^{\infty}U(X_i)$ is equal to $$\sum\limits_{\substack{(i_1,i_2,i_3) \in \mathbb N^2 \\ i_1,i_2,i_3 \text{ are distinct}}}X_{i_1}X_{i_2}X_{i_3} = 3! e_3 = X_G.$$
\end{example}

We claim that this is true for the graph of an arbitrary Borcherds algebra in the following proposition. The proof is exactly the same as the steps in the above two examples.
\begin{prop}\label{helprop}
		\textit{Let $G$ be the graph of a Borcherds algebra $\lie g$. 	Assume that the set $I$ is finite, i.e., $\lie g$ is of finite rank. We set $\eta(\mathbf{1})=\sum_{i\in I}\alpha_i\in Q_+$. Then 
	$$   \Big(\prod_{i = 1}^{\infty}U_1(X_i) \Big)[e^{-\eta(\bold 1)}] = (-1)^{\htt(\eta(\bold 1))}\  X_G.$$}
	
\end{prop}
\begin{proof}
	The required coefficient is equal to 
	$$\sum_{k = 1}^{\infty}\sum_{\substack{J \subseteq \mathbb{N}\\J = \{i_1,i_2,\dots,i_k\}}}\sum_{((w_1,\gamma_1),(w_2,\gamma_2),\dots,(w_k,\gamma_k)) \in (W \times \Omega)^{k}}(-1)^{\sum_{i=1}^k\htt(\gamma_i)}(-1)^{\ell(w_1\cdots w_k)} \prod_{j=1}^{k}\Big(X_{i_j}^{\ell(w_j)+\htt(\gamma_j)}\Big)$$
	where the sum ranges over all $k$--tuples $((w_1,\gamma_1),(w_2,\gamma_2),\dots,(w_k,\gamma_k)) \in (W \times \Omega)^{k}$ such that
	\begin{align*}
	&\bullet (w_i,\gamma_i) \text{ is stable for all } 1\le i\le k,&\\&
	\bullet \ I(w_1)\ \dot{\cup} \cdots \dot{\cup}\ I(w_k)=I^{\mathrm{re}},&\\&
	\bullet \ I(\gamma_1)\ \dot{\cup} \cdots \dot{\cup}\ I(\gamma_k)=I^{\mathrm{im}},&\\&
	\bullet \ I(w_i)\cup I(\gamma_i)\neq \emptyset \ \text{ for each $1\le i\le k$},&\\&
	\end{align*}
	
	It follows that $\big(I(w_1)\cup I(\gamma_1),\dots,I(w_k)\cup I(\gamma_k)\big)\in P_k(\bold 1,G)$ and each element is obtained in this way. So the sum ranges over all elements in $P_k(\bold 1,G).$
	Since $w_1\cdots w_k$ is a subword of a Coxeter element we get
	$$(-1)^{\ell(w_1\cdots w_k)}=(-1)^{|I^{\mathrm{re}}|},$$
	and hence $\Big(\prod_{i = 1}^{\infty}U(X_i)\Big)[e^{-\eta(\bold k)}]$ is equal to $$(-1)^{\htt(\eta(\bold 1))}\sum_{k \ge 1} \sum_{\substack{\mathcal{P} \in P_{k}(\bold 1,G)\\ \mathcal{P} = (P_1,P_2, \dots, P_k) }} \sum_{\substack{J \subseteq \mathbb{N}\\ J = \{i_1,i_2,\dots,i_k\}}} x_{i_1}^{|P_1|} x_{i_2}^{|P_2|} \cdots x_{i_k}^{|P_k|}$$ 
	Now, Proposition \ref{eqdef} finishes the proof.
\end{proof}
\begin{rem}
Assume that the set $I$ is not finite in the above theorem. Then consider $\bold k = (k_i:i \in I) \in \mathbb Z_+^{I}$ with finite support such that $k_i=1$ for $i \in \supp \bold k$. We can apply the above theorem for this $\bold k$ by replacing the graph $G$ by the subgraph of $G$ generated by $\supp \bold k$.
\end{rem}

\subsection{The product side of the denominator identity}\label{prod}
In this subsection, using the product side of the modified Weyl denominator identity, we will derive an expression for chromatic symmetric function in terms of root multiplicities of Borcherds algebras. 
Proposition~\ref{helprop} and Equation \eqref{mdenom} together imply that the chromatic symmetric function $X_G$ is given by the coefficient of $e^{-\eta(\bold 1)}$ in 
$$(-1)^{\htt(\eta(\bold 1))}\prod_{i=1}^{\infty}\prod_{\alpha\in \Delta_+}(1-X_i^{\htt(\alpha)}e^{-\alpha})^{\text{ dim }\mathfrak{g}_{\alpha}}.$$
This product is equal to,
$$(-1)^{\htt(\eta(\bold 1))}\prod_{i=1}^{\infty}\prod_{\alpha\in \Delta_+}\Big(1-(\mult \alpha) X_i^{\htt(\alpha)}e^{-\alpha})+\binom{\mult \alpha}{2}X_i^{\htt(2\alpha)}e^{-2\alpha}-\cdots\Big).$$
Consider the lattice $\mathcal{P}(\bold 1)$ [c.f. Lemma \ref{bij}]. 
Then, the coefficient of 
$e^{-\eta(\bold 1)}$ in $$(-1)^{\htt(\eta(\bold 1))}\prod_{i=1}^{\infty}\prod_{\alpha\in \Delta_+}\Big(1-(\mult \alpha) X_i^{\htt(\alpha)}e^{-\alpha})+ \text{ higher order terms}\Big)$$ is equal to $$  (-1)^{\htt(\eta(\bold 1))}\sum_{\substack{P \in \mathcal{P}(\bold 1) \\  P = \{\alpha_1,\dots,\alpha_k\}}} (-1)^k\mult(P)\sum_{\substack{J \subseteq \mathbb N \\ J = \{i_1 ,i_2, \dots, i_k\}}}    \prod_{\alpha_j \in P }\Big( X_{i_j}^{\htt (\alpha_j)}\Big)$$ where $\mult P = \prod_{\alpha_i \in P} \mult \alpha_i$. We can defined multiplicity of an element of $L_G$ using the bijection $\Psi$ defined in Lemma \ref{bij}. Now, Lemma \ref{bij} finishes the proof of Theorem \ref{mainthm}.


\begin{table}
	\caption{Graphs of order 4 are distinguished by the chromatic symmetric functions}
	\label{t1}
	\begin{tabular}{|c|c|c|c|c|c|}
		\hline
		S.No. &	\text{Graphs} & $(4)$ & $(2,1,1)$ & $(3,1)$ &  $(2,2)$ \\ 
		\hline

		1&\tikzset{every picture/.style={line width=0.75pt}} 
		
		\begin{tikzpicture}[x=0.75pt,y=0.75pt,yscale=-1,xscale=1]
		
		\draw  [fill={rgb, 255:red, 0; green, 0; blue, 0 }  ,fill opacity=1 ] (91,100.25) .. controls (91,98.46) and (92.46,97) .. (94.25,97) .. controls (96.04,97) and (97.5,98.46) .. (97.5,100.25) .. controls (97.5,102.04) and (96.04,103.5) .. (94.25,103.5) .. controls (92.46,103.5) and (91,102.04) .. (91,100.25) -- cycle ;
		\draw  [fill={rgb, 255:red, 0; green, 0; blue, 0 }  ,fill opacity=1 ] (181,100.25) .. controls (181,98.46) and (182.46,97) .. (184.25,97) .. controls (186.04,97) and (187.5,98.46) .. (187.5,100.25) .. controls (187.5,102.04) and (186.04,103.5) .. (184.25,103.5) .. controls (182.46,103.5) and (181,102.04) .. (181,100.25) -- cycle ;
		\draw  [fill={rgb, 255:red, 0; green, 0; blue, 0 }  ,fill opacity=1 ] (151,100.25) .. controls (151,98.46) and (152.46,97) .. (154.25,97) .. controls (156.04,97) and (157.5,98.46) .. (157.5,100.25) .. controls (157.5,102.04) and (156.04,103.5) .. (154.25,103.5) .. controls (152.46,103.5) and (151,102.04) .. (151,100.25) -- cycle ;
		\draw  [fill={rgb, 255:red, 0; green, 0; blue, 0 }  ,fill opacity=1 ] (121,100.25) .. controls (121,98.46) and (122.46,97) .. (124.25,97) .. controls (126.04,97) and (127.5,98.46) .. (127.5,100.25) .. controls (127.5,102.04) and (126.04,103.5) .. (124.25,103.5) .. controls (122.46,103.5) and (121,102.04) .. (121,100.25) -- cycle ;

		\end{tikzpicture}
		& 0 & 0 & 0 & 0 \\
		\hline

		2&\tikzset{every picture/.style={line width=0.75pt}} 
		
		\begin{tikzpicture}[x=0.75pt,y=0.75pt,yscale=-1,xscale=1]
		
		\draw  [fill={rgb, 255:red, 0; green, 0; blue, 0 }  ,fill opacity=1 ] (91,100.25) .. controls (91,98.46) and (92.46,97) .. (94.25,97) .. controls (96.04,97) and (97.5,98.46) .. (97.5,100.25) .. controls (97.5,102.04) and (96.04,103.5) .. (94.25,103.5) .. controls (92.46,103.5) and (91,102.04) .. (91,100.25) -- cycle ;
		\draw  [fill={rgb, 255:red, 0; green, 0; blue, 0 }  ,fill opacity=1 ] (181,100.25) .. controls (181,98.46) and (182.46,97) .. (184.25,97) .. controls (186.04,97) and (187.5,98.46) .. (187.5,100.25) .. controls (187.5,102.04) and (186.04,103.5) .. (184.25,103.5) .. controls (182.46,103.5) and (181,102.04) .. (181,100.25) -- cycle ;
		\draw  [fill={rgb, 255:red, 0; green, 0; blue, 0 }  ,fill opacity=1 ] (151,100.25) .. controls (151,98.46) and (152.46,97) .. (154.25,97) .. controls (156.04,97) and (157.5,98.46) .. (157.5,100.25) .. controls (157.5,102.04) and (156.04,103.5) .. (154.25,103.5) .. controls (152.46,103.5) and (151,102.04) .. (151,100.25) -- cycle ;
		\draw  [fill={rgb, 255:red, 0; green, 0; blue, 0 }  ,fill opacity=1 ] (121,100.25) .. controls (121,98.46) and (122.46,97) .. (124.25,97) .. controls (126.04,97) and (127.5,98.46) .. (127.5,100.25) .. controls (127.5,102.04) and (126.04,103.5) .. (124.25,103.5) .. controls (122.46,103.5) and (121,102.04) .. (121,100.25) -- cycle ;
		\draw    (154.25,100.25) -- (184.25,100.25) ;

		\end{tikzpicture}
		&0  & -1 & 0 & 0 \\
		\hline

		3&\tikzset{every picture/.style={line width=0.75pt}} 
		
		\begin{tikzpicture}[x=0.75pt,y=0.75pt,yscale=-1,xscale=1]
		
		\draw  [fill={rgb, 255:red, 0; green, 0; blue, 0 }  ,fill opacity=1 ] (91,100.25) .. controls (91,98.46) and (92.46,97) .. (94.25,97) .. controls (96.04,97) and (97.5,98.46) .. (97.5,100.25) .. controls (97.5,102.04) and (96.04,103.5) .. (94.25,103.5) .. controls (92.46,103.5) and (91,102.04) .. (91,100.25) -- cycle ;
		\draw  [fill={rgb, 255:red, 0; green, 0; blue, 0 }  ,fill opacity=1 ] (181,100.25) .. controls (181,98.46) and (182.46,97) .. (184.25,97) .. controls (186.04,97) and (187.5,98.46) .. (187.5,100.25) .. controls (187.5,102.04) and (186.04,103.5) .. (184.25,103.5) .. controls (182.46,103.5) and (181,102.04) .. (181,100.25) -- cycle ;
		\draw  [fill={rgb, 255:red, 0; green, 0; blue, 0 }  ,fill opacity=1 ] (151,100.25) .. controls (151,98.46) and (152.46,97) .. (154.25,97) .. controls (156.04,97) and (157.5,98.46) .. (157.5,100.25) .. controls (157.5,102.04) and (156.04,103.5) .. (154.25,103.5) .. controls (152.46,103.5) and (151,102.04) .. (151,100.25) -- cycle ;
		\draw  [fill={rgb, 255:red, 0; green, 0; blue, 0 }  ,fill opacity=1 ] (121,100.25) .. controls (121,98.46) and (122.46,97) .. (124.25,97) .. controls (126.04,97) and (127.5,98.46) .. (127.5,100.25) .. controls (127.5,102.04) and (126.04,103.5) .. (124.25,103.5) .. controls (122.46,103.5) and (121,102.04) .. (121,100.25) -- cycle ;
		\draw    (154.25,100.25) -- (184.25,100.25) ;
		\draw    (124.25,100.25) -- (154.25,100.25) ;

		\end{tikzpicture}
		&0   &   & 1  & 0  \\
		\hline

		4&\tikzset{every picture/.style={line width=0.75pt}} 
		
		\begin{tikzpicture}[x=0.75pt,y=0.75pt,yscale=-1,xscale=1]
		
		\draw  [fill={rgb, 255:red, 0; green, 0; blue, 0 }  ,fill opacity=1 ] (91,100.25) .. controls (91,98.46) and (92.46,97) .. (94.25,97) .. controls (96.04,97) and (97.5,98.46) .. (97.5,100.25) .. controls (97.5,102.04) and (96.04,103.5) .. (94.25,103.5) .. controls (92.46,103.5) and (91,102.04) .. (91,100.25) -- cycle ;
		\draw  [fill={rgb, 255:red, 0; green, 0; blue, 0 }  ,fill opacity=1 ] (181,100.25) .. controls (181,98.46) and (182.46,97) .. (184.25,97) .. controls (186.04,97) and (187.5,98.46) .. (187.5,100.25) .. controls (187.5,102.04) and (186.04,103.5) .. (184.25,103.5) .. controls (182.46,103.5) and (181,102.04) .. (181,100.25) -- cycle ;
		\draw  [fill={rgb, 255:red, 0; green, 0; blue, 0 }  ,fill opacity=1 ] (151,100.25) .. controls (151,98.46) and (152.46,97) .. (154.25,97) .. controls (156.04,97) and (157.5,98.46) .. (157.5,100.25) .. controls (157.5,102.04) and (156.04,103.5) .. (154.25,103.5) .. controls (152.46,103.5) and (151,102.04) .. (151,100.25) -- cycle ;
		\draw  [fill={rgb, 255:red, 0; green, 0; blue, 0 }  ,fill opacity=1 ] (121,100.25) .. controls (121,98.46) and (122.46,97) .. (124.25,97) .. controls (126.04,97) and (127.5,98.46) .. (127.5,100.25) .. controls (127.5,102.04) and (126.04,103.5) .. (124.25,103.5) .. controls (122.46,103.5) and (121,102.04) .. (121,100.25) -- cycle ;
		\draw    (154.25,100.25) -- (184.25,100.25) ;
		\draw    (97.5,100.25) -- (113.75,100.25) -- (127.5,100.25) ;

		\end{tikzpicture}
		&0  &  &  & 1  \\
		\hline

		5&\tikzset{every picture/.style={line width=0.75pt}} 
		
		\begin{tikzpicture}[x=0.75pt,y=0.75pt,yscale=-1,xscale=1]
		
		\draw  [fill={rgb, 255:red, 0; green, 0; blue, 0 }  ,fill opacity=1 ] (105,142.25) .. controls (105,140.46) and (106.46,139) .. (108.25,139) .. controls (110.04,139) and (111.5,140.46) .. (111.5,142.25) .. controls (111.5,144.04) and (110.04,145.5) .. (108.25,145.5) .. controls (106.46,145.5) and (105,144.04) .. (105,142.25) -- cycle ;
		\draw  [fill={rgb, 255:red, 0; green, 0; blue, 0 }  ,fill opacity=1 ] (120,160.25) .. controls (120,158.46) and (121.46,157) .. (123.25,157) .. controls (125.04,157) and (126.5,158.46) .. (126.5,160.25) .. controls (126.5,162.04) and (125.04,163.5) .. (123.25,163.5) .. controls (121.46,163.5) and (120,162.04) .. (120,160.25) -- cycle ;
		\draw  [fill={rgb, 255:red, 0; green, 0; blue, 0 }  ,fill opacity=1 ] (91,160.25) .. controls (91,158.46) and (92.46,157) .. (94.25,157) .. controls (96.04,157) and (97.5,158.46) .. (97.5,160.25) .. controls (97.5,162.04) and (96.04,163.5) .. (94.25,163.5) .. controls (92.46,163.5) and (91,162.04) .. (91,160.25) -- cycle ;
		\draw  [fill={rgb, 255:red, 0; green, 0; blue, 0 }  ,fill opacity=1 ] (151,160.25) .. controls (151,158.46) and (152.46,157) .. (154.25,157) .. controls (156.04,157) and (157.5,158.46) .. (157.5,160.25) .. controls (157.5,162.04) and (156.04,163.5) .. (154.25,163.5) .. controls (152.46,163.5) and (151,162.04) .. (151,160.25) -- cycle ;
		\draw    (108.25,142.25) -- (123.25,159.25) ;
		\draw    (94.25,160.25) -- (123.25,160.25) ;
		\draw    (94.25,160.25) -- (108.25,142.25) ;

		\end{tikzpicture}
		& 0  &  & 2  & 0  \\
		\hline

		6&

		\tikzset{every picture/.style={line width=0.75pt}} 
		
		\begin{tikzpicture}[x=0.75pt,y=0.75pt,yscale=-1,xscale=1]
		
		\draw  [fill={rgb, 255:red, 0; green, 0; blue, 0 }  ,fill opacity=1 ] (291.75,100.25) .. controls (291.75,98.46) and (293.21,97) .. (295,97) .. controls (296.79,97) and (298.25,98.46) .. (298.25,100.25) .. controls (298.25,102.04) and (296.79,103.5) .. (295,103.5) .. controls (293.21,103.5) and (291.75,102.04) .. (291.75,100.25) -- cycle ;
		\draw  [fill={rgb, 255:red, 0; green, 0; blue, 0 }  ,fill opacity=1 ] (318.75,100.25) .. controls (318.75,98.46) and (320.21,97) .. (322,97) .. controls (323.79,97) and (325.25,98.46) .. (325.25,100.25) .. controls (325.25,102.04) and (323.79,103.5) .. (322,103.5) .. controls (320.21,103.5) and (318.75,102.04) .. (318.75,100.25) -- cycle ;
		\draw  [fill={rgb, 255:red, 0; green, 0; blue, 0 }  ,fill opacity=1 ] (318.75,127.25) .. controls (318.75,125.46) and (320.21,124) .. (322,124) .. controls (323.79,124) and (325.25,125.46) .. (325.25,127.25) .. controls (325.25,129.04) and (323.79,130.5) .. (322,130.5) .. controls (320.21,130.5) and (318.75,129.04) .. (318.75,127.25) -- cycle ;
		\draw  [fill={rgb, 255:red, 0; green, 0; blue, 0 }  ,fill opacity=1 ] (291.75,127.25) .. controls (291.75,125.46) and (293.21,124) .. (295,124) .. controls (296.79,124) and (298.25,125.46) .. (298.25,127.25) .. controls (298.25,129.04) and (296.79,130.5) .. (295,130.5) .. controls (293.21,130.5) and (291.75,129.04) .. (291.75,127.25) -- cycle ;
		\draw    (295,100.25) -- (295,127.25) ;
		\draw    (322,100.25) -- (322,127.25) ;
		\draw    (295,127.25) -- (323,127.25) ;

		\end{tikzpicture}
		
		& -1 &  & 2 &  \\
		\hline

		\tikzset{every picture/.style={line width=0.75pt}} 
		
		7&\begin{tikzpicture}[x=0.75pt,y=0.75pt,yscale=-1,xscale=1]
		
		\draw  [fill={rgb, 255:red, 0; green, 0; blue, 0 }  ,fill opacity=1 ] (205,137.25) .. controls (205,135.46) and (206.46,134) .. (208.25,134) .. controls (210.04,134) and (211.5,135.46) .. (211.5,137.25) .. controls (211.5,139.04) and (210.04,140.5) .. (208.25,140.5) .. controls (206.46,140.5) and (205,139.04) .. (205,137.25) -- cycle ;
		\draw  [fill={rgb, 255:red, 0; green, 0; blue, 0 }  ,fill opacity=1 ] (265,137.25) .. controls (265,135.46) and (266.46,134) .. (268.25,134) .. controls (270.04,134) and (271.5,135.46) .. (271.5,137.25) .. controls (271.5,139.04) and (270.04,140.5) .. (268.25,140.5) .. controls (266.46,140.5) and (265,139.04) .. (265,137.25) -- cycle ;
		\draw  [fill={rgb, 255:red, 0; green, 0; blue, 0 }  ,fill opacity=1 ] (235,137.25) .. controls (235,135.46) and (236.46,134) .. (238.25,134) .. controls (240.04,134) and (241.5,135.46) .. (241.5,137.25) .. controls (241.5,139.04) and (240.04,140.5) .. (238.25,140.5) .. controls (236.46,140.5) and (235,139.04) .. (235,137.25) -- cycle ;
		\draw    (211.5,137.25) -- (227.75,137.25) -- (241.5,137.25) ;
		\draw    (238.25,137.25) -- (268.25,137.25) ;
		\draw    (238.25,137.25) -- (238.25,165.25) ;
		\draw  [fill={rgb, 255:red, 0; green, 0; blue, 0 }  ,fill opacity=1 ] (235,165.25) .. controls (235,163.46) and (236.46,162) .. (238.25,162) .. controls (240.04,162) and (241.5,163.46) .. (241.5,165.25) .. controls (241.5,167.04) and (240.04,168.5) .. (238.25,168.5) .. controls (236.46,168.5) and (235,167.04) .. (235,165.25) -- cycle ;

		\end{tikzpicture}
		& -1  &  & 3 &  \\
		\hline

		8&\tikzset{every picture/.style={line width=0.75pt}} 
		
		\begin{tikzpicture}[x=0.75pt,y=0.75pt,yscale=-1,xscale=1]
		
		\draw  [fill={rgb, 255:red, 0; green, 0; blue, 0 }  ,fill opacity=1 ] (277.76,136.51) .. controls (279.55,136.52) and (281,137.97) .. (281,139.77) .. controls (281,141.56) and (279.54,143.02) .. (277.75,143.01) .. controls (275.95,143.01) and (274.5,141.55) .. (274.5,139.76) .. controls (274.5,137.96) and (275.96,136.51) .. (277.76,136.51) -- cycle ;
		\draw  [fill={rgb, 255:red, 0; green, 0; blue, 0 }  ,fill opacity=1 ] (259.73,151.49) .. controls (261.53,151.49) and (262.98,152.95) .. (262.98,154.74) .. controls (262.98,156.54) and (261.52,157.99) .. (259.73,157.99) .. controls (257.93,157.98) and (256.48,156.53) .. (256.48,154.73) .. controls (256.48,152.94) and (257.94,151.48) .. (259.73,151.49) -- cycle ;
		\draw  [fill={rgb, 255:red, 0; green, 0; blue, 0 }  ,fill opacity=1 ] (259.77,122.49) .. controls (261.57,122.49) and (263.02,123.95) .. (263.02,125.74) .. controls (263.02,127.54) and (261.56,128.99) .. (259.77,128.99) .. controls (257.97,128.98) and (256.52,127.53) .. (256.52,125.73) .. controls (256.52,123.94) and (257.98,122.48) .. (259.77,122.49) -- cycle ;
		\draw  [fill={rgb, 255:red, 0; green, 0; blue, 0 }  ,fill opacity=1 ] (302,140.25) .. controls (302,138.46) and (303.46,137) .. (305.25,137) .. controls (307.04,137) and (308.5,138.46) .. (308.5,140.25) .. controls (308.5,142.04) and (307.04,143.5) .. (305.25,143.5) .. controls (303.46,143.5) and (302,142.04) .. (302,140.25) -- cycle ;
		\draw    (277.75,139.76) -- (260.73,154.74) ;
		\draw    (259.77,125.74) -- (259.73,154.74) ;
		\draw    (259.77,125.74) -- (277.75,139.76) ;
		\draw    (277.75,139.76) -- (305.25,140.25) ;

		\end{tikzpicture}
		& -2  &  &  &  \\
		\hline

		9&\tikzset{every picture/.style={line width=0.75pt}} 
		
		\begin{tikzpicture}[x=0.75pt,y=0.75pt,yscale=-1,xscale=1]
		
		\draw   (224,105) -- (251.5,105) -- (251.5,132.5) -- (224,132.5) -- cycle ;
		\draw  [fill={rgb, 255:red, 0; green, 0; blue, 0 }  ,fill opacity=1 ] (220.75,105.25) .. controls (220.75,103.46) and (222.21,102) .. (224,102) .. controls (225.79,102) and (227.25,103.46) .. (227.25,105.25) .. controls (227.25,107.04) and (225.79,108.5) .. (224,108.5) .. controls (222.21,108.5) and (220.75,107.04) .. (220.75,105.25) -- cycle ;
		\draw  [fill={rgb, 255:red, 0; green, 0; blue, 0 }  ,fill opacity=1 ] (247.75,105.25) .. controls (247.75,103.46) and (249.21,102) .. (251,102) .. controls (252.79,102) and (254.25,103.46) .. (254.25,105.25) .. controls (254.25,107.04) and (252.79,108.5) .. (251,108.5) .. controls (249.21,108.5) and (247.75,107.04) .. (247.75,105.25) -- cycle ;
		\draw  [fill={rgb, 255:red, 0; green, 0; blue, 0 }  ,fill opacity=1 ] (248.75,132.25) .. controls (248.75,130.46) and (250.21,129) .. (252,129) .. controls (253.79,129) and (255.25,130.46) .. (255.25,132.25) .. controls (255.25,134.04) and (253.79,135.5) .. (252,135.5) .. controls (250.21,135.5) and (248.75,134.04) .. (248.75,132.25) -- cycle ;
		\draw  [fill={rgb, 255:red, 0; green, 0; blue, 0 }  ,fill opacity=1 ] (220.75,132.25) .. controls (220.75,130.46) and (222.21,129) .. (224,129) .. controls (225.79,129) and (227.25,130.46) .. (227.25,132.25) .. controls (227.25,134.04) and (225.79,135.5) .. (224,135.5) .. controls (222.21,135.5) and (220.75,134.04) .. (220.75,132.25) -- cycle ;

		\end{tikzpicture}
		& -3  &  &  &  \\
		\hline

		10&\tikzset{every picture/.style={line width=0.75pt}} 
		
		\begin{tikzpicture}[x=0.75pt,y=0.75pt,yscale=-1,xscale=1]
		
		\draw   (298,121) -- (325.5,121) -- (325.5,148.5) -- (298,148.5) -- cycle ;
		\draw  [fill={rgb, 255:red, 0; green, 0; blue, 0 }  ,fill opacity=1 ] (294.75,121.25) .. controls (294.75,119.46) and (296.21,118) .. (298,118) .. controls (299.79,118) and (301.25,119.46) .. (301.25,121.25) .. controls (301.25,123.04) and (299.79,124.5) .. (298,124.5) .. controls (296.21,124.5) and (294.75,123.04) .. (294.75,121.25) -- cycle ;
		\draw  [fill={rgb, 255:red, 0; green, 0; blue, 0 }  ,fill opacity=1 ] (321.75,121.25) .. controls (321.75,119.46) and (323.21,118) .. (325,118) .. controls (326.79,118) and (328.25,119.46) .. (328.25,121.25) .. controls (328.25,123.04) and (326.79,124.5) .. (325,124.5) .. controls (323.21,124.5) and (321.75,123.04) .. (321.75,121.25) -- cycle ;
		\draw  [fill={rgb, 255:red, 0; green, 0; blue, 0 }  ,fill opacity=1 ] (322.75,148.25) .. controls (322.75,146.46) and (324.21,145) .. (326,145) .. controls (327.79,145) and (329.25,146.46) .. (329.25,148.25) .. controls (329.25,150.04) and (327.79,151.5) .. (326,151.5) .. controls (324.21,151.5) and (322.75,150.04) .. (322.75,148.25) -- cycle ;
		\draw  [fill={rgb, 255:red, 0; green, 0; blue, 0 }  ,fill opacity=1 ] (294.75,148.25) .. controls (294.75,146.46) and (296.21,145) .. (298,145) .. controls (299.79,145) and (301.25,146.46) .. (301.25,148.25) .. controls (301.25,150.04) and (299.79,151.5) .. (298,151.5) .. controls (296.21,151.5) and (294.75,150.04) .. (294.75,148.25) -- cycle ;
		\draw    (298,121) -- (326,148.25) ;

		\end{tikzpicture}
		& -4 &  &  &  \\
		
		\hline	
		
		11&\tikzset{every picture/.style={line width=0.75pt}} 
		
		\begin{tikzpicture}[x=0.75pt,y=0.75pt,yscale=-1,xscale=1]
		
		\draw   (298,121) -- (325.5,121) -- (325.5,148.5) -- (298,148.5) -- cycle ;
		\draw  [fill={rgb, 255:red, 0; green, 0; blue, 0 }  ,fill opacity=1 ] (294.75,121.25) .. controls (294.75,119.46) and (296.21,118) .. (298,118) .. controls (299.79,118) and (301.25,119.46) .. (301.25,121.25) .. controls (301.25,123.04) and (299.79,124.5) .. (298,124.5) .. controls (296.21,124.5) and (294.75,123.04) .. (294.75,121.25) -- cycle ;
		\draw  [fill={rgb, 255:red, 0; green, 0; blue, 0 }  ,fill opacity=1 ] (321.75,121.25) .. controls (321.75,119.46) and (323.21,118) .. (325,118) .. controls (326.79,118) and (328.25,119.46) .. (328.25,121.25) .. controls (328.25,123.04) and (326.79,124.5) .. (325,124.5) .. controls (323.21,124.5) and (321.75,123.04) .. (321.75,121.25) -- cycle ;
		\draw  [fill={rgb, 255:red, 0; green, 0; blue, 0 }  ,fill opacity=1 ] (322.75,148.25) .. controls (322.75,146.46) and (324.21,145) .. (326,145) .. controls (327.79,145) and (329.25,146.46) .. (329.25,148.25) .. controls (329.25,150.04) and (327.79,151.5) .. (326,151.5) .. controls (324.21,151.5) and (322.75,150.04) .. (322.75,148.25) -- cycle ;
		\draw  [fill={rgb, 255:red, 0; green, 0; blue, 0 }  ,fill opacity=1 ] (294.75,148.25) .. controls (294.75,146.46) and (296.21,145) .. (298,145) .. controls (299.79,145) and (301.25,146.46) .. (301.25,148.25) .. controls (301.25,150.04) and (299.79,151.5) .. (298,151.5) .. controls (296.21,151.5) and (294.75,150.04) .. (294.75,148.25) -- cycle ;
		\draw    (298,121) -- (326,148.25) ;
		\draw    (325,121.25) -- (298,148.5) ;

		\end{tikzpicture}
		& -6  &  &  & \\
		\hline

	\end{tabular}
	
\end{table}

\begin{example}\label{G4}
In Table \ref{t1}, we have shown that the graphs of order $4$ are distinguished by the chromatic symmetric functions using Theorem \ref{mainthm}. We remark that the coefficients of $p_{\lambda}$ appearing in $X_G$ whichever is required to distinguish the graphs of order $4$ are given in Table \ref{t1}. These values are calculated using the following multiplicity formula from \cite[Corollary 3.9]{akv}. Assume $\bold k = (k_1,k_2,\dots,) \in \mathbb Z_+^{I}$ with finite support and satisfies $k_i \le 1$ for $i \in I^{re}$. Then
\begin{equation}\label{mult}
	\mult \eta(\bold k) = \sum\limits_{l | \bold k} \frac{\mu(l)}{l}|\pi^G_{\frac{\bold k}{l}}(q)[q]|.
\end{equation} 	where $\pi^G_{\frac{\bold k}{l}}(q)[q]$ denotes the coefficient of $q$ in $\pi^G_{\frac{\bold k}{l}}(q)$.


\end{example}
\begin{example}
	Consider $\lie g = \lie{sl}_3(\mathbb C)$ as in Example \ref{ssl3}. In this case, $\mathcal P(\bold 1) = \{\{\alpha_1+\alpha_2\},\{\alpha_1,\alpha_2\}\}$ and all the roots have multiplicity one [c.f. Lemma \ref{bij}]. The product side of the modified denominator identity of $\lie g$ is as follows. 
	$$U(X_i) = \prod_{\alpha\in \Delta_+}(1-X_i^{-\htt(\alpha)}e^{-\alpha})^{\text{ dim }\mathfrak{g}_{\alpha}} = (1-X_i^{-\htt(\alpha_1)}e^{-\alpha_1})(1-X_i^{-\htt(\alpha_1)}e^{-\alpha_2})(1-X_i^{-\htt(\alpha)}e^{-\alpha_1-\alpha_2}).$$
	We have  $\mult \,\{\alpha_1+\alpha_2\} = \mult \,\{\alpha_1,\alpha_2\} = 1$ as an element of $\mathcal P(\bold 1)$ where multiplicity of a set in $P(\bold 1)$ is defined to be the product of the multiplicity of its elements. Note that, the map $\Psi$ defined in Lemma \ref{bij} respects the multiplicity. Now, $$(-1)^2\sum_{\bold J \in \mathcal P(\bold 1)} (-1)^{|\bold J|} (\mult (\bold J)) \,p_{\type (\bold J)} = -p_{(2)} + p_{(1,1)}  = 2\,e_2 = X_G.$$

\end{example}
\begin{example}\label{ex9}
Consider $\lie g = \widehat{\lie{sl}_3(\mathbb C)}$ as in Example \ref{ssl3h}. In this case, $\Delta_+ = \Delta_+^0 \sqcup (\sqcup_{k>0}(\Delta^0 + k\delta)) \sqcup \{k\delta : k > 0\}$ where $\Delta^0$ is the root system of the underlying $\lie{sl}_3$ and $\delta$ is the null root. More precisely, $\Delta^0_+ = \{\alpha_1,\alpha_2,\alpha_1+\alpha_2\}$ and $\delta = \alpha_0 + \alpha_1 + \alpha_2$. Further, $\delta$ has multiplicity 2. We have 
	
	$\mathcal P(\bold 1) =\{ \{\alpha_0+\alpha_1+\alpha_2\}, \{{\alpha_0},\alpha_1+\alpha_2\}, \{\alpha_0+\alpha_1,\alpha_2\},\{\alpha_0+\alpha_2,\alpha_1\}, \{\alpha_0,\alpha_1,\alpha_2\}\}$.
	
	Now, $$(-1)^3\sum_{\bold J \in \mathcal P(\bold 1)} (-1)^k (\mult (\bold J)) \,p_{\type (\bold J)} = 2p_{(3)} - 3p_{(2,1)} + p_{(1^3)} = 3! e_3 = X_G$$
\end{example}

\begin{rem}
	Consider the bond lattice $L_G$ of weight $\bold 1$. Let $S \subseteq E(G)$. Consider the graph $G_S$ formed by the vertex set $I$ and the edge set $S$. Let $P_S=\{P_1,\dots,P_k\}$ be the vertex sets of the connected components of the graph $G_S$. Then $P_S \in L_G$.  Conversely let $P = \{P_1,\dots,P_k\} \in L_G$ and let $G(P_i)$ be the subgraph of $G$ induced by the set $P_i$. Then each $G(P_i)$ is a connected subgraph of $G$. Let $E_i$ be the edge set of $G(P_i)$ and $S_P := \cup_{i=1}^kE_i \subseteq E(G)$. Let $\sigma$ be the map from $\mathcal P(E(G))$ to $L_G$ which maps $S$ to $P_S$. We observe that $c(S)$-the number of components of the spanning subgraph $G_S$ is equal to $|\sigma(S)|$-the number of parts in the partition $\sigma(S) \in L_G$. We have the following well-known expression for chromatic polynomial due to Whitney \cite{whitney32}:
	\begin{equation}\label{whitney}
		\pi_{\bold 1}^G(q) = \sum\limits_{S \subseteq E}(-1)^{|S|}q^{c(S)}.
	\end{equation}  By Equation \ref{vv} we have $$\pi_{\bold 1}^G(q) = \sum_{\bold J \in L_G} (-1)^{l - |\bold J|} (\mult (\bold J)) \,q^{|\bold J|}$$ where $l := \htt (\eta(\bold 1))$. The above two equations implies that for a fixed $\bold J \in L_G$, \begin{equation}\label{con}
		\sum\limits_{\substack{S \subseteq E \\ \sigma(S) = \bold J}}(-1)^{|S|} = (-1)^{l-|\bold J|} \mult \bold J.
	\end{equation}
	
	In particular, if $\bold J$ is the partition with single part consists of the full vertex set $I$ of $G$ then we have \begin{equation}\label{4.3}
		\sum\limits_{\substack{S \subseteq E \\ \sigma(S) = \bold J}}(-1)^{|S|} = (-1)^{l-1} \mult \bold J = (-1)^{l-1} \mult (\eta(\bold 1)).
	\end{equation}
	
	Now, Theorem \ref{mainthm} and Equation \eqref{con} implies the following expression for the chromatic symmetric function \cite[Theorem 2.5]{S95}.
	
	$$X_G = \sum\limits_{S \subseteq E}(-1)^{|S|}p_{\type S}$$ where $\type S = \type P_S$. Hence if we apply the above arguments in Section \ref{prod}, last two lines above Example \ref{t1}, we get the above expression for the chromatic symmetric function from the denominator identity.
\end{rem}

\begin{rem}
	An edge cover of a graph $G$ is a set $S$ of edges such that every vertex of $G$ is incident to at least one edge of the set $S$. See \cite{akbari13,MR2838026,MR1099076} for combinatorial results on edge cover of graphs. If $\bold J$ is the partition with a single part consists of the full vertex set of $G$ then from Equation \eqref{4.3} we have $$	\sum\limits_{\substack{S \subseteq E \\ \psi(S) = \bold J}}(-1)^{|S|} = (-1)^{l-1} \mult \bold 1.$$ We observe that the subsets $S \subseteq E$ satisfying $\psi(S) = \bold J$ are precisely the edge covers of $G$. Now, Equation \eqref{4.3} gives an interesting formula for the multiplicity of the root $\eta(\bold 1)$ in terms of edge covers of $G$. For example, consider the null root $\delta = \alpha_0 + \alpha_1 + \alpha_2$ given in Example \ref{ex9}. The edge covers of the cycle graph on three vertices $\{\alpha_0,\alpha_1,\alpha_2\}$ are $\{\{\alpha_0,\alpha_1,\alpha_2\},\{\alpha_0,\alpha_1\},\{\alpha_0,\alpha_2\},\{\alpha_1,\alpha_2\}\}$. This shows that the multiplicity of the null root is equal to $2$. 

\end{rem}
Next, we prove Theorem \ref{mainkthm}.
\begin{rem}
	By Equation \eqref{1=k} and Theorem \ref{mainthm}, it is possible to recover the $\bold k$-chromatic symmetric function of $G$ from the modified denominator identity: More precisely, the $\bold k$-chromatic symmetric function of $G$ is given by the coefficient of $\frac{1}{\bold k!}e^{-\eta(\bold 1)}$ in the product $\prod_{i=1}^{\infty}U(X_i)$ where $U(X_i)$ is the modified denominator identity of the Borcherds algebra $\lie g(\bold k)$ corresponds to the graph $G(\bold k)$.  But it is possible to recover the $\bold k$-chromatic symmetric function from the modified denominator identity of $\lie g$ itself. We explain this fact in the following example.
\end{rem}

\begin{example}
	Let $\lie g$ be the Borcherds algebra corresponds to the $1 \times 1$ order Borcherds-Cartan matrix $\begin{bmatrix}
	-1 
	\end{bmatrix}$. Then $\lie g$ is generated by the three elements $<h,e,f>$ subject to the relations. $[h,h] =0$, $[h,e]=-e$ and $[h,f] = f$. The root system of $\lie g$ is given by $\Delta = \{\alpha,-\alpha\}$, the Weyl group $W = \{e\}$ and $\Omega = \{0,\alpha\}$. Therefore $$U(X_i) = 1 - X_i e^{-\alpha}.$$ Let $\bold k = (k)$, $k \in \mathbb N$. Consider the product
	$$\prod_{i=1}^{\infty} U(X_i) = (1 - X_1 e^{-\alpha})(1 - X_2 e^{-\alpha})\cdots(1 - X_k e^{-\alpha})\cdots$$ 
	The coefficient of $e^{-k\alpha}$ in the above product is equal to
	$$\sum_{\substack{J \,\subseteq \mathbb{N}\\J = \{i_1,i_2,\dots,i_k\}}}(-1)^k(X_{i_1}\cdots X_{i_k}) = (-1)^ke_k = (-1)^{\htt (\bold k)}X_G^{\bold k}.$$
\end{example}

Let $\lie g$ be a Borcherds algebra with the associated graph $G$. In view of the above example, we see that the coefficient of $e^{-\eta(\bold k)}$ in $\prod_{i=1}^{\infty}U(X_i)$ is equal to $(-1)^{\htt (\bold k)}X^{\bold k}_G$. The proof is exactly the same as the proof of Proposition \ref{helprop}. Now, an argument similar to the one given in Section \ref{prod} completes the proof of Theorem \ref{mainkthm}. 

\section{$G$-symmetric  functions and the denominator identity of Borcherds algebras}\label{G-syms}

\subsection{} In this subsection, we define the stable part of the modified denominator identity and prove that the stable part of the modified denominator identity is the same as the $G$-analogues of the elementary symmetry function. 


Let $\lie g$ be a Borcherds algebra with the associated graph $G$ and Weyl group $W$. For a Weyl group element $w\in W$,
we fix a reduced word $w=\bold {s}_{i_1}\cdots \bold{s}_{i_k}$ and let $I(w)=\{\alpha_{i_1},\dots,\alpha_{i_k}\}$. 
We recall that $I(w)$ is independent of the choice of the reduced expression of $w$. For $\gamma\in \Omega$, we set $I(\gamma)=\{\alpha\in \Pi^{\mathrm im} : \mbox{ $\alpha$ is a summand of $\gamma$}\}$. 
A pair $(w,\gamma) \in W \times \Omega$ is said to be stable if $I(w) \cup I(\gamma)$ is a stable set and we define $l((w,\gamma)) = \ell(w) + \htt(\gamma)$.  Let $\mathcal{S}(G)$ be the set of all stable subsets of $G$. It is clear that there exists a bijection $f$ between the set $\mathcal{S}(G)$ and $\{(w,\gamma) \in W \times \Omega : (w,\gamma) \text{ is stable}\}$. Further this map satisfies $l(f(S)) = |S|$.  Given this definition, the Weyl denominator can be rewritten as 
\begin{eqnarray}\label{stablepart}
U&=&  \sum_{(w,\gamma) \in W \times \Omega}  \epsilon(w,\gamma)  e^{w(\rho -\gamma)-\rho },  \\
&=& \sum_{\substack{(w,\gamma) \in W \times \Omega \\ \text{stable} }} \epsilon(w,\gamma)  e^{w(\rho -\gamma)-\rho } + \sum_{\substack{(w,\gamma) \in W \times \Omega \\ \text{not stable} }} \epsilon(w,\gamma)  e^{w(\rho -\gamma)-\rho }, \\
&=& U_1 + U_2 \,\,\,(\text{say}),
\end{eqnarray}  
where $\epsilon(w,\gamma) = (-1)^{\ell(w)}(-1)^{\htt(\gamma)}$. 

We define $U_1$ to the stable part of the Weyl denominator $U$. The stable part $U_1(X)$ of the modified Weyl denominator $U(X)$ is defined similarly. 
We want to find a relation between $e^G_i$ and the stable part of the Weyl denominator $U$ [c.f. Definition \ref{G-symd}]. First, we will investigate the terms $e^{w(\rho-\gamma)-\rho}$ that are appearing in $U_1$. Now, from Lemma \ref{helplem}, it is easy to see that  
\begin{equation}\label{U=E}
	U_1 = \sum_{k \ge 0}^{\alpha(G)} \sum_{\substack{(w,\gamma) \in W \times \Omega \\ \text{stable} \\ l((w,\gamma)) = k }}(-1)^k e^{w(\rho-\gamma)-\rho}= \sum_{k \ge 0}^{\alpha(G)}\sum_{\substack{S\text{-stable} \\ |S|=k }}(-1)^ke(S) =  \sum_{k \ge 0}^{\alpha(G)}(-1)^ke_k^G,
\end{equation}
where $\alpha(G)$ is the independence number of $G$ and $e(S) = \prod_{\alpha \in S}e^{-\alpha}$. In particular, we have $$  e_k^G = \sum_{\substack{(w,\gamma) \in W \times \Omega \\ \text{stable} \\ l((w,\gamma)) = k }}e^{w(\rho-\gamma)-\rho}.$$  
This shows that $e_k^G$ can be obtained from the Weyl denominator.

\begin{rem}
	Given a partition $\lambda = \lambda_1 \ge \lambda_2 \ge \dots \ge \lambda_k$, $k \in \mathbb{N}$, we have $$e_{\lambda}^G= \prod_{i=1}^k e_{\lambda_i}^G = \prod_{i=1}^k \Big(\sum_{\substack{(w,\gamma) \in W \times \Omega \\ \text{stable} \\ l((w,\gamma)) = \lambda_i }}e^{w(\rho-\gamma)-\rho}\Big) = \sum_{\substack{((w_1,\gamma_1),(w_2,\gamma_2),\dots,(w_k,\gamma_k)) \in (W \times \Omega)^k \\ (w_i,\gamma_i)\text{-stable} \\ l((w_i,\gamma_i)) = \lambda_i }}e^{w(\rho-\gamma)-\rho} = \sum_{\substack{(S_1,S_2,\dots,Sk) \\ S_i\text{-stable} \\ |S_i| = \lambda_i }}e(S).$$

	 Suppose $(w,\gamma) \in W \times \Omega$ is stable, then $|I(w) \cup I(\gamma)| = \ell(w)+\htt(\gamma) = -\htt(w(\rho-\gamma)-\gamma)$. Again, by Lemma \ref{helplem}, it is easy to see that $$U_1(X) = \sum_{i \ge 0} X^i e_i^G.$$
\end{rem}

\begin{defn}
	A partition $\lambda = (\lambda_1,\dots,\lambda_k)$ is said to be a stable number partition of $G$ if $1 \le \lambda_i \le \alpha(G)$ for all $1 \le i \le k$. The set of all stable number partition of the graph $G$ is denoted by $SP(G)$.  
	
\end{defn}

\begin{prop}\label{U=T}
	With the notations as above, we have
$$\prod_{i = 1}^{\infty}U_1(X_i) = \sum_{\substack {\lambda \\ \text{stable}}} \epsilon(\lambda) M_{\lambda}(x) e_{\lambda}^G $$ where $\epsilon(\lambda) = (-1)^{\sum_{i=1}^k \lambda_i}$ for a partition $\lambda = (\lambda_1,\dots,\lambda_k)$.  
\begin{pf} 
\begin{align*}
\prod_{i = 1}^{\infty}U_1(X_i) &= \prod_{i = 1}^{\infty}  \Big(\sum_{\substack{(w,\gamma) \in W \times \Omega \\ \text{stable}} } (-1)^{(\ell(w)+\htt(\gamma))}  X_i^{\htt(w(\rho -\gamma)-\rho)} e^{w(\rho -\gamma)-\rho } \Big)
\\&:= \sum_{k = 1}^{\infty}  \sum_{\substack{J \,\subseteq\, \mathbb{N}\\J = \{i_1,i_2,\dots,i_k\}}} \prod_{j = 1}^{k}  \Big(\sum_{\substack{(w,\gamma) \in W \times \Omega \\ \text{stable}}} (-1)^{(\ell(w)+\htt(\gamma))}  X_{i_j}^{\htt(w(\rho -\gamma)-\rho)} e^{w(\rho -\gamma)-\rho } \Big)
\\&= \sum_{k = 1}^{\infty}  \sum_{\substack{J \,\subseteq\, \mathbb{N}\\J = \{i_1,i_2,\dots,i_k\}}} \sum_{\substack{((w_1,\gamma_1),(w_2,\gamma_2),\dots,(w_k,\gamma_k)) \in (W \times \Omega)^{k} \\ (w_i,\gamma_i)-\text{stable}}}  \prod_{j=1}^{k}\Big((-1)^{(\ell( w_j)+\htt(\gamma_j) }X_{i_j}^{\htt(w_j(\rho -\gamma_j)-\rho)}e^{w_j(\rho -\gamma_j)-\rho }\Big)
\end{align*}
From Lemma~\ref{helplem}, we get 
$$w(\rho)-\rho-w(\gamma)=-\gamma-\sum_{\alpha\in I(w)}\alpha \text{ and $\htt(w(\rho)-\rho-w(\gamma)) = \ell(w) + \htt(\gamma)$ if $(w,\gamma)$ is stable}.$$ 

This implies that,
\begin{align*}
\prod_{i = 1}^{\infty}U_1(X_i) &= \sum_{k = 1}^{\infty}  \sum_{\substack{J \,\subseteq \mathbb{N}\\J = \{i_1,i_2,\dots,i_k\}}} \sum_{\substack{((w_1,\gamma_1),(w_2,\gamma_2),\dots,(w_k,\gamma_k)) \in (W \times \Omega)^{k} \\ (w_i,\gamma_i)-\text{stable}}}  \prod_{j=1}^{k}\Big((-1)^{\ell( w_j)+\htt(\gamma_j) }X_{i_j}^{\ell(w_j) + \htt(\gamma_j)}e(S_j)\Big) 
\end{align*}
where $S_j = I(w_j) \cup I(\gamma_j)$ and $e(S_i) = e^{-\sum_{\alpha\in S_j}\alpha}$. 
Using the above defined bijection $f$ between the set $\mathcal{S}(G)$ and $\{(w,\gamma) \in W \times \Omega : (w,\gamma) \text{ is stable}\}$ we get,
\begin{align*}
&\text{Right hand side of the above equation}= \sum_{k = 1}^{\infty}  \sum_{\substack{J \,\subseteq \mathbb{N}\\J = \{i_1,i_2,\dots,i_k\}}} \sum_{\substack{(S_1,S_2,\dots,S_k) \\ S_i-\text{stable}}}  \prod_{j=1}^{k}\Big((-X_{i_j})^{|S_j|}e(S_j)\Big), \\
& = \sum_{k = 1}^{\infty}  \sum_{\substack{(S_1,S_2,\dots,S_k) \\ |S_1| \ge |S_2| \ge \dots \ge |S_k| \\ S_i-\text{stable}}}  \sum_{\substack{(S_{m_1},S_{m_2},\dots,S_{m_k}) \\ \text{permutation of }(S_1,S_2,\dots,S_k)}}  \sum_{\substack{J \,\subseteq \mathbb{N}\\J = \{i_1,i_2,\dots,i_k\}}} \Big(e(S_{m_1},S_{m_2},\dots,S_{m_k})\Big)  \Big(\prod_{j=1}^{k}(-X_{i_j})^{|S_{m_j}|}\Big), \\
&= \sum_{\substack{\lambda \\ \text{stable}}} \epsilon(\lambda) M_{\lambda}(x) e_{\lambda}^G,
\end{align*}
where $e(S_1,S_2,\dots,S_k) = \prod_{i=1}^k e(S_i)$. 
\end{pf}
\begin{rem}
 Propositions \ref{U=T} and \ref{helprop} together leads to an alternate proof of Proposition \ref{T=X}.
\end{rem}
\end{prop}

\subsection{$G$-power sum symmetric functions and the modified Weyl denominators}
In this subsection, we prove Theorem \ref{power}. We start with the definition of $G$-power sum symmetric function which is defined in terms of $G$-elementary symmetric functions. The following relation is proved in \cite{S98}. 
\begin{equation}\label{e=p}
	-\log(1-e_1^GX+e_2^GX^2-e^G_3X^3+\cdots) = p_1^GX + p_2^G\frac{X^2}{2} + p_3^G\frac{X^3}{3}+\cdots
\end{equation}
The $G$ analogues of power sum symmetric functions are defined using the above equation. To prove Theorem \ref{power}, it is enough to prove it for $p_n^G$ ($n \in \mathbb{N}$).  Let $\lie g = \lie g(-A)$ where $A$ is the adjacency matrix of $G$. Then all the simple roots of $\lie g$ are imaginary and the modified denominator identity \eqref{mdenom} of $\lie g$ becomes 
	\begin{equation}\label{free}
	U:= \sum_{\gamma \in
		\Omega} (-1)^{-\htt(\gamma)} X^{-\htt(\gamma)} e^{-\gamma}  = \prod_{\alpha \in \Delta_+} (1 - X^{-\htt(\alpha)}  e^{-\alpha})^{\dim \mathfrak{g}_{\alpha}}. 
	\end{equation}
We observe that, since all the simple roots are imaginary, the stable part $U_1$ of $U$ is itself.	Equation \eqref{U=E} and the remark which follows Equation \eqref{U=E}  implies that  $$U_1(-X) = \sum_{i \ge 0} (-X)^i e_i^G.$$
From Equation \eqref{e=p}, the coefficient of $\frac{X^n}{n}$ in $-\log(U(-X))$ is equal to $p_n^G$. Now, we calculate the same coefficient using the product side of Equation \eqref{free}. 
\begin{align*}
	-\log \Big(\prod_{\alpha \in \Delta_+} (1 - X^{-\htt(\alpha)}  e^{-\alpha})^{\dim \mathfrak{g}_{\alpha}}\Big) &= \sum_{ \alpha \in \Delta_+} \dim \lie g_{\alpha} \Big(-\log(1 - X^{-\htt(\alpha)}  e^{-\alpha})\Big), \\
	&= \sum_{ \alpha \in \Delta_+} \dim \lie g_{\alpha} \Big(\sum_{k \ge 1}\frac{(X^{-\htt(\alpha)}  e^{-\alpha})^k}{k}\Big),\\
	&= \sum_{k \ge 1} \sum_{m \ge 1} \sum_{\substack{\alpha \in \Delta_+ \\ \htt(\alpha) = m}} ((m) (\dim \lie g_{\alpha})  (e^{-k \alpha} )) \frac{X^{m k}}{m k}.
\end{align*}
Hence, the coefficient of $\frac{X^n}{n}$ in $-\log(U(X))$ is equal to 
\begin{align*}
\sum_{k | n} \Big( \sum_{ \substack{\alpha \in \Delta_+ \\ \htt(\alpha) = \frac{n}{k}}}(\frac{n}{k}) (\dim \lie g_{\alpha} )(e^{-k \alpha})\Big) &= \sum_{k | n} \Big( \sum_{ \substack{\frac{\alpha}{k} \in \Delta_+ \\ \htt(\frac{\alpha}{k}) = \frac{n}{k}}}(\frac{n}{k}) (\dim \lie g_{\frac{\alpha}{k}} )(e^{- \alpha})\Big) \\
& = \sum_{ \substack{\alpha \in \Delta_+ \\ \htt(\alpha) = n}}\Big(\sum_{k | \alpha}  (\frac{n}{k}) (\dim \lie g_{\frac{\alpha}{k}} )\Big)e^{- \alpha}\\
&= \sum_{k \ge 1} \sum_{m \ge 1} \sum_{\substack{\alpha \in \Delta_+ \\ \htt(\alpha) = m}} ((m) (\dim \lie g_{\alpha})  (e^{-k \alpha} )) \frac{X^{m k}}{m k}.
\end{align*}
This shows that $$p_n^G = \sum_{ \substack{\alpha \in \Delta_+ \\ \htt(\alpha) = n}}\Big(\sum_{k | \alpha}  (\frac{n}{k}) (\dim \lie g_{\frac{\alpha}{k}} )\Big)e^{- \alpha}$$ and the theorem follows.

\begin{rem}
	
	We have shown that the chromatic symmetric function can be recovered from the modified denominator identity. It is possible to recover the chromatic symmetric function from the denominator identity itself: Consider $U = \sum_{w \in W } (-1)^{\ell(w)} \sum_{ \gamma \in
		\Omega} (-1)^{\htt(\gamma)} e^{w(\rho -\gamma)-\rho }$ the sum side of the denominator identity given in Equation \ref{denominator}. By Lemma \ref{helplem},   $$e^{w(\rho -\gamma)-\rho } = \prod_{\alpha \in I}X_{\alpha}^{b_{\alpha}(w,\gamma)} \in \mathbb C[X_{\alpha}: \alpha \in I]$$ where $X_{\alpha} = e^{-\alpha}$. For a stable subset $S$ of $G$, we define $X_S:=\prod_{\alpha \in S}X_{\alpha} \in \mathbb C[X_{\alpha}: \alpha \in I]$. Then the stable part of $U = U_1 = \sum\limits_{S \text{-stable}}X_S$ [c.f. Equation \eqref{stablepart}]. 
	
	We consider the subalgebra $\mathcal A:= \mathbb C[\prod_{i=1}^{m}X_{S_i} : \{S_1,\dots,S_m\}\text{ is a set of stable sets for some }m]$ of $\mathbb C[X_{\alpha} : \alpha \in I]$. Define a linear map $\mathcal F : \mathcal A \rightarrow \mathbb C[[X_i : i \in \mathbb N]]$ by: 
	
	
	$$\mathcal F (\prod_{i=1}^{m}X_{S_i}) = \sum\limits_{\substack{J \subseteq \mathbb N \\ J = \{i_1,\dots,i_m\}}}\prod_{i=1}^mX_{i_1}^{|S_i|}$$ for any set of stable sets $\{S_1,\dots,S_m\}$. 
	Then, $$\mathcal F (\sum_{m \ge 1}U_1^m) = \mathcal F(\sum_{m \ge 1}\sum\limits_{\substack{(S_1,\dots,S_m) \\ S_i \text{-stable}}} \prod_{i=1}^{m}X_{S_i}) = \sum_{m \ge 1}\sum\limits_{\substack{(S_1,\dots,S_m) \\ S_i \text{-stable}}} \mathcal F(\prod_{i=1}^{m} X_{S_i}) =\sum_{m \ge 1}\sum\limits_{\substack{(S_1,\dots,S_m) \\ S_i \text{-stable}}} \sum\limits_{\substack{J \subseteq \mathbb N \\ J = \{i_1,\dots,i_m\}}}\prod_{i=1}^mX_{i_1}^{|S_i|}.$$
	Now, it is immediate that, $$\sum_{m \ge 1}\sum\limits_{\substack{(S_1,\dots,S_m) \\ \text{Stable partition of }G}} \mathcal F(\prod_{i=1}^{m} X_{S_i}) = \sum_{m \ge 1}\sum\limits_{\substack{(S_1,\dots,S_m) \\ \text{Stable partition of }G}} \sum\limits_{\substack{J \subseteq \mathbb N \\ J = \{i_1,\dots,i_m\}}}\prod_{i=1}^mX_{i_1}^{|S_i|} = X_G$$ 
	This way we can get the chromatic symmetric function from the denominator identity (Equation \eqref{denominator}). But this approach cannot be implemented in the product side of the denominator identity. So we prefer the approach through the modified denominator identity.  
\end{rem}

We discuss here a little about the future directions of this paper. In this paper, we have made a connection between chromatic symmetric functions and Borcherds algebras. In this direction, we can make a connection between chromatic symmetric functions and various other branches. The advantage of making such a connection is the tools from different areas can be applied in the problems of the theory of chromatic symmetric functions.

\begin{enumerate}
	\item 	Example \ref{G4} shows that root multiplicities of Borcherds algebras can be used to distinguish graphs by their chromatic symmetric functions. Given a graph $G$ we can associate a Borcherds algebra to it in many ways [c.f. Section \ref{graph}]. It is instructive to check whether given class of non-isomorphic graphs of order $n$ are distinguished by their chromatic symmetric functions using the root multiplicities of the associated Borcherds algebras (using Theorem \ref{mainthm}). 
	
	\item 	Borcherds algebras have a super analogue which is called Borcherds-Kac-Moody Lie superalgebras (BKM algebras in short) \cite{waki01}. The main difference here is that we are allowed to have odd roots. These algebras also have a similar but more complicated denominator identity \cite[Section 2.5]{waki01}. We can prove a connection between the root multiplicities of BKM algebras and the chromatic symmetric function of the associated graph $G$. In a different project under progress, we have already proved such connection for $\bold k$-chromatic polynomials which is further used construct basis for the root spaces of BKM superalgebras using heaps of pieces. 
	
	\item 	Heaps of pieces were introduced by Xavier Viennot in \cite{MR1110852} which has many applications in various branches of Mathematics and Physics. In \cite{a3} we have proved the various connection between the heaps of pieces and $\bold k$-chromatic polynomials. In particular, the heap theoretic analogue of the denominator identity of free partially commutative Lie algebra is discussed. We can further study this connection in the level of chromatic symmetric functions which will give us a connection between the heaps of pieces and the chromatic symmetric functions. 
	
	\item 	Also, we can reprove the existing results in the chromatic symmetric function theory using our Lie theoretic methods. This way the tools from Borcherds algebras can be used in the chromatic symmetric function theory.
	
	\item 	In the introduction we have discussed the relation between the Macdonald identities and the denominator formula of affine Lie algebras. We can look for similar results in the case of the modified denominator identity (Equation \eqref{mdenom}).
\end{enumerate}

\bibliographystyle{plain}
\bibliography{Arunkumar}

\end{document}